\documentclass[a4paper,11pt]{amsart}

\usepackage[all]{xy}

\usepackage[utf8]{inputenc}		
\usepackage[T1]{fontenc}
\usepackage[english]{babel}

\usepackage{amsfonts}			
\usepackage{amsmath}
\usepackage{amssymb}
\usepackage{amsthm}
\usepackage{mathtools}

\usepackage{graphicx}			
\usepackage{tikz}
	\usetikzlibrary{cd}

\usepackage{hyperref}			
	\hypersetup{colorlinks=false, urlcolor=black, linkcolor=black}

\usepackage{enumitem}			

\usepackage{color}				

\newcommand{\N}{\mathbb{N}}						
\newcommand{\Z}{\mathbb{Z}}						
\newcommand{\R}{\mathbb{R}}						

\renewcommand{\S}{\mathbb{S}}					

\newcommand{\Julia}{\mathcal{J}}				
\newcommand{\Fatou}{\mathcal{F}}				

\newcommand{\eps}{\varepsilon}					

\newcommand{\dd}								
	{\mathop{}\!\mathrm{d}}						
\newcommand{\ddn}[1]							
	{\mathop{}\!\mathrm{d^{#1}}}

\newcommand{\abs}[1]							
	{\left| #1 \right|}
\newcommand{\smallabs}[1]						
	{\lvert #1 \rvert}	
\newcommand{\norm}[1]							
	{\left\lVert #1 \right\rVert}	
\newcommand{\smallnorm}[1]						
	{\lVert #1 \rVert}						
\newcommand{\ip}[2]								
	{\left< #1 , #2 \right>}

\DeclareMathOperator{\id}{id}					
\DeclareMathOperator{\proj}{pr}					
\DeclareMathOperator{\intr}{int}				
\DeclareMathOperator{\len}{len}
\DeclareMathOperator{\Span}{span}
\DeclareMathOperator{\aff}{aff}					

\DeclareMathOperator{\tran}{T}				
\DeclareMathOperator{\ort}{O}

\DeclareMathOperator{\im}{im}					

\newcommand{\loc}{\mathrm{loc}}

\newcommand{\cE}{\mathcal{E}}
\newcommand{\cG}{\mathcal{G}}
\newcommand{\cH}{\mathcal{H}}

\newtheorem{thm}{Theorem}[section]{\bf}{\it}
\newtheorem{lemma}[thm]{Lemma}

\newtheorem{cor}[thm]{Corollary}

\newtheorem{innercustomthm}{Theorem}[section]{\bf}{\it}
\newenvironment{customthm}[1]
	{\renewcommand\theinnercustomthm{#1}\innercustomthm}
	{\endinnercustomthm}
	
{\bf}{\it}

\theoremstyle{definition}
\newtheorem{defn}[thm]{Definition}

\theoremstyle{remark}

\numberwithin{equation}{section}

\begin{document}

\title[Obstructions for automorphic and Latt\`es maps]{Obstructions for automorphic quasiregular maps and Latt\`es-type uniformly quasiregular maps}
\author{Ilmari Kangasniemi}
\address{Department of Mathematics and Statistics, P.O. Box 68 (Pietari Kalmin Katu 5), FI-00014 University of Helsinki, Finland}
\email{ilmari.kangasniemi 'at' helsinki.fi}

\thanks{This work was supported by the doctoral program DOMAST of the University of Helsinki and the Academy of Finland project \#297258.}

\begin{abstract}
	Suppose that $M$ is a closed, connected, and oriented Riemannian $n$-manifold, $f \colon \R^n \to M$ is a quasiregular map automorphic under a discrete group $\Gamma$ of Euclidean isometries, and $f$ has finite multiplicity in a fundamental cell of $\Gamma$. We show that if $\Gamma$ has a sufficiently large translation subgroup $\Gamma_T$, then $\dim \Gamma \in \{0, n-1, n\}$. If $f$ is strongly automorphic and induces a non-injective Latt\`es-type uniformly quasiregular map, then the same assertion holds without the assumption on the size of $\Gamma_T$. Moreover, an even stronger restriction holds in the Latt\`es case if $M$ is not a rational cohomology sphere.
\end{abstract}

\maketitle

\section{Introduction}

A continuous map $f \colon M \to N$ between oriented Riemannian $n$-manifolds is \emph{$K$-quasiregular} for $K \geq 1$ if $f$ belongs to the local Sobolev space $W^{1, n}_\loc (M, n)$ and $\abs{Df(x)}^n \leq K J_f(x)$ for almost every $x \in M$. Given a quasiregular map $f \colon \R^n \to M$, an element $\omega \in \R^n$ is a \emph{period} of $f$ if $f(x + \omega) = f(x)$ for all $x \in \R^n$, and $f$ is \emph{$k$-periodic} for $k \in \{1, \dots, n\}$ if its periods span a $k$-dimensional subspace of $\R^n$.

In \cite{Martio1975paper_kperiod}, Martio proved the following obstruction result on $k$-periodic quasiregular maps.

\begin{thm}[{\cite[Theorem 1.1]{Martio1975paper_kperiod}}]\label{thm:Martios_result}
	Let $f \colon \R^n \to M$ be a $k$-periodic quasiregular map, where $k > 0$ and $M$ is either $\R^n$ or $\S^n$. Suppose that $k \leq n-2$. Then $f$ has infinite multiplicity in its period strip.
\end{thm}

In this paper, we investigate generalizations of Theorem \ref{thm:Martios_result} for quasiregular maps which are automorphic with respect to a discrete group of Euclidean isometries. Recall that if $\Gamma$ is a group acting on $\R^n$, then a map $f \colon \R^n \to M$ is \emph{automorphic} with respect to $\Gamma$ if $f \circ \gamma = f$ for every $\gamma \in \Gamma$. Both $k$-periodic and automorphic quasiregular mappings have seen significant amounts of study, with a large portion of it being due to Martio and Srebro; see eg.\ \cite{Martio1975paper_kperiod}, \cite{MartioSrebro1975paper2}, \cite{MartioSrebro1975paper1}, and \cite{MartioSrebro1999paper}. In addition, automorphic quasiregular maps have seen application in the definition of Latt\`es-type uniformly quasiregular maps, which have been studied eg.\ in \cite{Mayer1997paper}, \cite{Mayer1998paper}, \cite{MartinMayer2003paper}, \cite{Astola-Kangaslampi-Peltonen}, and \cite{FletcherMacclure2017preprint}.

For our first main result, we prove an automorphic generalization of Theorem \ref{thm:Martios_result}. However, the result requires an extra assumption on the dimension of the translation subgroup of $\Gamma$. 

\begin{thm}\label{thm:generalized_martio_periodicity_result}
	Let $M$ be a closed oriented Riemannian $n$-manifold, and let $f \colon \R^n \to M$ be quasiregular. Suppose that $h$ is automorphic with respect to a discrete subgroup $\Gamma$ of the isometry group $E(n)$ of $\R^n$, and let $\Gamma_T$ denote the subgroup of translations of $\Gamma$. If $0 < \dim \Gamma < n$ and
	\begin{equation}\label{eq:dimension_limitation}
		\frac{\dim \Gamma_T}{\dim \Gamma} > \frac{1}{n-\dim \Gamma},
	\end{equation}
	then $f$ has infinite multiplicity in a fundamental cell of $\Gamma$.
\end{thm}

We define a \emph{fundamental cell} of $\Gamma$ to be a connected set $D \subset \R^n$ such that $D$ contains exactly one point from every orbit of $\Gamma$ and $\partial D$ has Lebesgue measure zero. If $f$ is automorphic with respect to $\Gamma$, then $f$ has the same multiplicity in every fundamental cell of $\Gamma$.

Moreover, if $\Gamma \leqslant E(n)$ is discrete, the \emph{dimension} $\dim \Gamma$ of $\Gamma$ is the largest $k \in \N$ for which $\Gamma$ contains an isomorphic copy of $\Z^k$; a more precise exposition is given in Section \ref{subsect:dimension_and_growth}. It is worth noting that there exist discrete $\Gamma \leqslant E(n)$ without a $\dim \Gamma$-dimensional translation subgroup; see Section \ref{subsect:translation_subgroup} for details.

Our proof of Theorem \ref{thm:generalized_martio_periodicity_result} follows Martio's proof of Theorem \ref{thm:Martios_result}, which in turn is based on a construction of Rickman \cite{Rickman1975paper}. The fact that the target space is more general than just $\R^n$ or $\S^n$ is handled by a simple application of a bilipschitz chart, and ends up causing very little difference in the proof. However, changing $k$-periodicity into automorphicness breaks Martio's method of obtaining a crucial length estimate. We recover the result in our case by a slightly more refined method of obtaining the length estimate, but our estimate is sufficient to follow Martio's proof only when \eqref{eq:dimension_limitation} holds. 

We remark that \eqref{eq:dimension_limitation} in fact implies that $\dim \Gamma \leq n-2$, which more closely mirrors the assumption in Martio's original result. In particular, if $\dim \Gamma_T = \dim \Gamma$, then \eqref{eq:dimension_limitation} is equivalent with $\dim \Gamma \leq n-2$. However, due to the exceptionality of the scenario which would cause the proof to fail, we conjecture that the condition \eqref{eq:dimension_limitation} can be weakened to $\dim \Gamma \leq n-2$.

\medskip
\emph{Update:} After this paper had been available for some time in pre-print form, Sylvester Eriksson-Bique pointed out a way to greatly improve Theorem \ref{thm:generalized_martio_periodicity_result} when we were both visiting IM PAN in Warsaw. His improvement yields the following result, included in this paper with his permission.

\begin{thm}\label{thm:Sylvesters_version}
	Let $M$ be a closed oriented Riemannian $n$-manifold, and let $f \colon \R^n \to M$ be quasiregular. Suppose that $h$ is automorphic with respect to a discrete subgroup $\Gamma$ of the isometry group $E(n)$ of $\R^n$. If $0 < \dim \Gamma < n-2$, then $f$ has infinite multiplicity in a fundamental cell of $\Gamma$.
\end{thm}

Therefore, by combining Theorems \ref{thm:Sylvesters_version} and \ref{thm:generalized_martio_periodicity_result}, the remaining case where an automorphic generalization of Martio's Theorem \ref{thm:Martios_result} is unknown is when $\dim \Gamma = n-2$, and $\dim \Gamma_T \leq (\dim \Gamma)/2$. Eriksson-Bique's improvement is to use the fact that abelian groups of matrices are simultaneously diagonalizable; we explain the short argument in Section \ref{sect:sylvesters_improvement}.

\medskip

For the second main result, we consider Theorem \ref{thm:Martios_result} in the setting of Latt\`es-type uniformly quasiregular mappings, which we define as follows.

\begin{defn}
	Let $M$ be a closed, connected, and oriented Riemannian manifold of dimension $n$. We call a triple $(\Gamma, h, A)$ a \emph{Latt\`es triple into $M$} if
	\begin{itemize}
		\item $\Gamma$ is a discrete subgroup of $E(n)$;
		\item $h \colon \R^n \to M$ is a quasiregular map which is \emph{strongly automorphic} with respect to $\Gamma$: that is, $h$ is automorphic with respect to $\Gamma$, and if $h(x) = h(y)$, then $x = \gamma(y)$ for some $\gamma \in \Gamma$;
		\item $A \colon \R^n \to \R^n$ is a linear conformal bijection with $A \Gamma A^{-1} \subset \Gamma$.
	\end{itemize}
	A uniformly quasiregular map $g \colon M \to M$ is a \emph{Latt\`es map} if there exists a Latt\`es triple $(\Gamma, h, A)$ into $M$ for which $g \circ h = h \circ A$.
\end{defn}

Recall that a quasiregular self-map $g \colon M \to M$ is \emph{uniformly quasiregular} if there exists $K \geq 1$ for which every iterate $g^{j}$ of $g$ is $K$-quasiregular, $j > 0$. If $(\Gamma, h, A)$ is a Latt\`es triple, then it induces a uniformly quasiregular map $g \colon h(\R^n) \to h(\R^n)$ for which $g \circ h = h \circ A$: see e.g. Iwaniec--Martin \cite[Theorem 21.4.1]{IwaniecMartin2001book} or Astola--Kangaslampi--Peltonen \cite[Theorem 2.3]{Astola-Kangaslampi-Peltonen}. A Latt\`es map is therefore an extension of such a map $g$ to the whole manifold $M$. Note that, for closed manifolds $M$, the Picard-type theorem of Holopainen and Rickman \cite{HolopainenRickman1992paper} implies that $M \setminus h(\R^n)$ is finite. 

For our second main result, we show that for Latt\`es triples, the problems caused by automorphicness in Theorem \ref{thm:generalized_martio_periodicity_result} are in fact completely avoided. Namely, while a general map can be strongly automorphic with respect to a discrete $\Gamma \leqslant E(n)$ without being $(\dim \Gamma)$-periodic, the other conditions of Latt\`es triples present an obstruction to this in practically all interesting situations. More precisely, we have the following.

\begin{thm}\label{prop:translations_under_A-condition}
	Let $\Gamma \leqslant E(n)$ be a discrete group of isometries of $\R^n$, and let $A \colon \R^n \to \R^n$ be a linear conformal map. Suppose that $A$ is expanding and $A \Gamma A^{-1} \subset \Gamma$. Then $\dim \Gamma_T = \dim \Gamma$, where $\Gamma_T$ is the translation subgroup of $\Gamma$.
\end{thm}

We note that something akin to the conclusion of Theorem \ref{prop:translations_under_A-condition} appears to be considered true in the literature; see Mayer \cite[Proposition 3.1]{Mayer1998paper}, where Martio's Theorem \ref{thm:Martios_result} is used to derive a restriction on the dimension of $\Gamma$ for Latt\`es triples $(\Gamma, h, A)$ into $\S^n$. Since Theorem \ref{thm:generalized_martio_periodicity_result} is given for closed manifold targets, this method of proof immediately gives a version of \cite[Proposition 3.1]{Mayer1998paper} for more general closed manifolds.

\begin{cor}\label{cor:dimension_restriction_for_manifolds}
	Let $M$ be a closed, connected, and oriented Riemannian $n$-manifold, and let $(\Gamma, h, A)$ be a Latt\`es triple into $M$. Suppose that $A$ is expanding. Then $\dim \Gamma \in \{0, n-1, n\}$.
\end{cor}

If $M$ is a closed, connected, and oriented Riemannian $n$-manifold, one may ask whether the induced uniformly quasiregular map $g \colon h(\R^n) \to h(\R^n)$ of a Latt\`es triple $(\Gamma, h, A)$ into $M$ always extends to a uniformly quasiregular $g \colon M \to M$. This was shown to be true for the most interesting case of $M = \S^n$ by Mayer \cite[Proposition 3.2]{Mayer1998paper}, although the proof appears to be written with the implicit assumption that $\dim \Gamma_T = \dim \Gamma$. We give a version of this extension result for all closed, connected, oriented Riemannian $n$-manifolds $M$. Note that we include the case $0 \leq \dim \Gamma \leq n-2$ for completeness, as we do not assume that $A$ is expanding.

\begin{thm}\label{prop:extension_result}
	Let $M$ be a closed, connected, and oriented Riemannian $n$-manifold, let $(\Gamma, h, A)$ be a Latt\`es triple into $M$ with induced uniformly quasiregular map $g \colon h(\R^n) \to h(\R^n)$, and let $k = \dim \Gamma$. Then we have the following three cases: 
	\begin{itemize}
		\item if $k = n$, then $h$ is surjective;
		\item if $k = n-1$, then $h$ omits either 1 or 2 points;
		\item if $0 \leq k \leq n-2$, then $h$ omits 1 and only 1 point.
	\end{itemize}
	Moreover, $g\colon h(\R^n) \to h(\R^n)$ always extends to a Latt\`es map $g \colon M \to M$ on the entire manifold $M$.
\end{thm}

The restriction on the number of points omitted by $h$ follows the ideas of Martio and Srebro  \cite[Theorem 8.2]{MartioSrebro1975paper1}, and is an easy byproduct of a version of \cite[Lemma 3.1]{Martio1975paper_kperiod} which we need for the proof of Theorem \ref{thm:generalized_martio_periodicity_result}. While the extension could be done using Mayer's original methods, we provide a simple alternative proof based on general extension results of quasiregular maps.

Finally, we observe that for some closed manifolds $M$, the topology of $M$ imposes further restrictions on the dimension of Latt\`es triples. It is proven in \cite[Theorem 1.2]{Kangasniemi2017paper} that non-injective uniformly quasiregular maps on closed manifolds with nontrivial rational cohomology have large Julia sets. By applying this to Latt\`es maps, we obtain the following result.

\begin{thm}\label{prop:surjectivity_based_on_cohomology}
	Let $M$ be a closed, connected, and oriented Riemannian $n$-manifold, and let $(\Gamma, h, A)$ be a Latt\`es triple into $M$. Suppose that $A$ is expanding and $M$ is not a rational cohomology sphere. Then $\dim \Gamma = n$ and $h$ is surjective.
\end{thm}

Consequently, under the assumptions of Theorem \ref{prop:surjectivity_based_on_cohomology}, the quotient map $\R^n / \Gamma \to h(\R^n) = M$ induced by $h$ is a homeomorphism. Therefore, $M$ is topologically a manifold which is homeomorphic to a flat orbifold. Note that this does not mean that $M$ is a flat manifold: there exists a flat orbifold structure on the space $S^2 \times S^2$, and moreover $S^2 \times S^2$ admits non-injective Latt\`es maps by \cite[Section 4.2]{Astola-Kangaslampi-Peltonen}, but $S^2 \times S^2$ does not admit a flat smooth Riemannian metric. See also \cite{Martin-Mayer-Peltonen} for related discussion.

\subsection{Structure of this paper}

Section \ref{sect:eucl_isom} is a review of basic facts and properties related to the group $E(n)$ of Euclidean isometries and its discrete subgroups. Afterwards, we first present the proofs for the Latt\`es-specific facts. In Section \ref{sect:lin_conf_interaction}, we study the condition $A \Gamma A^{-1} \subset \Gamma$, and obtain Theorem \ref{prop:translations_under_A-condition}. In Section \ref{sect:extension}, we show the extension part of Theorem \ref{prop:extension_result}. In Section \ref{sect:extra_proposition}, we prove Theorem \ref{prop:surjectivity_based_on_cohomology}. 

After Section \ref{sect:extra_proposition}, the rest of the paper focuses on the proof of Theorem \ref{thm:generalized_martio_periodicity_result}. In Section \ref{sect:limit_lemma}, we formulate a version of \cite[Lemma 3.1]{Martio1975paper_kperiod}, obtaining the remining part of Theorem \ref{prop:extension_result} as a consequence. Section \ref{sect:path_lifting_short} is on the automorphic variant of the path lifting lemma \cite[Lemma 4.2]{Martio1975paper_kperiod}. In Section \ref{sect:length_estimate}, we derive the length estimate used on the lifted paths, which is the main difference to the periodic case in the proof of Theorem \ref{thm:generalized_martio_periodicity_result}. Finally, in Sections \ref{sect:main_theorem} and \ref{sect:sylvesters_improvement}, we prove Theorems \ref{thm:generalized_martio_periodicity_result} and \ref{thm:Sylvesters_version}.

\subsection{Acknowledgments}

The author thanks his PhD thesis supervisor Pek\-ka Pankka for helpful critical comments on the paper, as well as Volker Mayer and Alastair Fletcher for e-mail discussions related to the subject. Moreover, the author thanks Sylvester Eriksson-Bique for pointing out the argument leading to Theorem \ref{thm:Sylvesters_version}, as well as for other helpful critical comments on the paper.

\section{Euclidean isometries}\label{sect:eucl_isom}

Let $E(n)$ denote the $n$-dimensional \emph{Euclidean group}, that is, the isometry group of $\R^n$. The purpose of this section is to review the necessary basics of $E(n)$ and its discrete subgroups. We assume that the results in this section are well-known, but regardless provide the proofs for convenience. For more discussion, see eg. Szczepanski \cite{Szczepanski2012book} or Wolf \cite[Chapter 3]{Wolf1967book}.

Recall that any isometry of $\R^n$ can be written in the form $x \mapsto Ax + a$, where $a \in \R^n$ and $A$ is an element of the $n$-dimensional \emph{orthogonal group} $O(n)$. Hence, $E(n)$ may be considered as a semidirect product $O(n) \rtimes \R^n$ with the composition rule $(A,a) \circ (B,b) = (AB, Ab + a)$ and inverse elements $(A, a)^{-1} = (A^{-1}, -A^{-1}a)$. We denote 
\[
	\tran (A, a) = a, \quad \ort (A, a) = A.
\]

The group $O(n)$ admits an invariant metric, where the distance of two elements $A, B \in O(n)$ is the operator norm of the linear map $A - B$. With respect to this metric, $O(n)$ is a compact topological group. We may then topologize $E(n)$ using the product topology.

Let $\Gamma$ be a discrete subgroup of $E(n)$. We note that $\Gamma$ is closed since $E(n)$ is a Hausdorff topological group. Moreover the subset 
\[
	\tran(\Gamma) = \{ \tran(\gamma) : \gamma \in \Gamma \} 
\]
of $\R^n$ is closed and discrete, since the map $\tran \colon E(n) \to \R^n$ is proper by the compactness of $O(n)$, and therefore $\tran$ maps closed discrete sets to closed discrete sets.

We remark here an important basic property of $O(n)$ as a compact group, which we formulate as a lemma since it is used multiple times in this paper. The proof is simple, but we recall it nonetheless for convenience. 

\begin{lemma}\label{lem:basic_prop_of_O(n)}
	Suppose $A \in O(n)$. Then there is a subsequence of $(A^m)_{m=1}^\infty$ which converges to $\id_{\R^n}$.
\end{lemma}
\begin{proof}
	Since $O(n)$ is compact, there is a subsequence $(A^{m_j})_{j=1}^\infty$ converging to an element $A' \in O(n)$. By moving to a further subsequence, we may assume that $m_{j+1} - m_j > m_j - m_{j-1}$ for every $j \geq 2$. It follows that $(A^{m_{j+1} - m_j})_{j=1}^\infty$ is a subsequence of $(A^m)_{m=1}^\infty$.
	
	We claim that $(A^{m_{j+1} - m_j})_{j=1}^\infty$ converges to $\id_{\R^n}$. Indeed, since $(A^{m_j})_{j=1}^\infty$ is convergent, it is Cauchy, and therefore for large enough $j$ we have 
	\[
		\norm{A^{m_{j+1} - m_j} - \id_{\R^n}} = \norm{A^{m_{j+1}} - A^{m_j}} < \eps.
	\]
	Hence, the claim follows.
\end{proof}

\subsection{Subgroup of translations}\label{subsect:translation_subgroup}

For a discrete subgroup $\Gamma$ of $E(n)$, we denote by $\Gamma_T$ the subgroup of translations of $\Gamma$. A major classical result of discrete Euclidean groups is Bieberbach's first theorem: if $\Gamma \leqslant E(n)$ is discrete and $\R^n / \Gamma$ is compact, then $\Gamma_T$ is of finite index in $\Gamma$. Moreover, in this case $\Gamma_T$ is $n$-dimensional, in a sense which is made precise in Section \ref{subsect:dimension_and_growth}.

However, if $\R^n / \Gamma$ is not compact, then $\Gamma_T$ may fail to be of finite index in $\Gamma$. A simple example of this is a discrete group $\Gamma$ of screw-motions in $\R^3$, consisting of the maps $\gamma_k \colon (r, \theta, z) \mapsto (r, \theta + k\theta_0, z + k)$ in cylindrical coordinates where $k \in \Z$ and $\theta_0/\pi$ is irrational. Note, however, that while $\Gamma$ has no translations, the elements of $\Gamma$ restrict to the $z$-axis as translations. We formulate this property as follows.

\begin{defn}\label{def:subspace_of_translations}
	Suppose that $\Gamma$ is a discrete subgroup of $E(n)$, $G$ is a subgroup of $\Gamma$, and $V$ is an affine subspace of $\R^n$. We call the pair $(G, V)$ a \emph{cocompact translation pair of $\Gamma$} if $G$ acts cocompactly on $V$ by translations; that is, $GV = V$, the restriction $g \vert_V$ is a translation for every $g \in G$, and the quotient space $G / V$ is compact. 
	
	Suppose that $(G, V)$ is a cocompact translation pair of $\Gamma$. If $G$ is of finite index in $\Gamma$, we say that $(G, V)$ is a \emph{finite index cocompact translation pair of $\Gamma$}. Furthermore, if $g \in G$, we denote by $\tran_V(g)$ the translation vector of $g\vert_V$.
\end{defn}

Now, an extension of Bieberbach's first theorem to the case where $\R^n / \Gamma$ is not compact is given by the following theorem of Wolf.

\begin{thm}[{\cite[Theorem 3.2.8]{Wolf1967book}}]\label{thm:isometry_group_fact}
	Let $\Gamma$ be a discrete subgroup of $E(n)$. Then there exists a finite index cocompact translation pair $(G, V)$ of $\Gamma$ for which $G$ is a normal Abelian subgroup of $\Gamma$ containing $\Gamma_T$.
\end{thm}

Before continuing, we point out some basic properties of cocompact translation pairs . The proofs are simple, but are still included for convenience. 

\begin{lemma}\label{lem:distance_preserving_fact}
	Let $\Gamma$ be a discrete subgroup of $E(n)$, and let $(G,V)$ be a cocompact translation pair of $\Gamma$. Then for every $x \in \R^n$ and $g \in G$, we have $d(x, V) = d(g(x), V)$.
\end{lemma}
\begin{proof}
	Let $g \in G$ and $x \in \R^n$. Since $V \subset \R^n$ is closed, there exists a $v \in V$ for which $d(x, V) = d(x, v)$. Since $GV = V$, we have $g(v) \in V$, and therefore
	\begin{equation}\label{eq:distance_estimate}
		d(x, V) = d(x, v) = d(g(x), g(v)) \leq d(g(x), V).
	\end{equation}
	For the opposite estimate, we apply \eqref{eq:distance_estimate} on $g^{-1} \in G$ and $g(x) \in \R^n$, obtaining
	\[
		d(g(x), V) \leq d(g^{-1}(g(x)), V) = d(x, V).
	\]
\end{proof}

\begin{lemma}\label{lem:identity_restriction}
	Let $\Gamma$ be a discrete subgroup of $E(n)$, let $(G,V)$ be a cocompact translation pair of $\Gamma$, and let $V'$ be the linear space parallel to $V$. Then for every $g \in G$, we have $\ort(g) \vert_{V'} = \id_{V'}$.
\end{lemma}
\begin{proof}
	Let $g \in G$ and $v' \in V'$. Select a point $x \in V$. Since $x + v' \in V$ and $g$ is a translation on $V$, we have $v' = (x+v') - x = g(x + v') - g(x)$. On the other hand, since $g$ is affine with linear part $\ort(g)$, we have $g(x + v') - g(x) = \ort(g)(x + v' - x) = \ort(g)v'$. We conclude that $\ort(g)v' = v'$, which yields the claim.
\end{proof}

\subsection{Dimension and growth}\label{subsect:dimension_and_growth}

We present two equivalent definitions for the dimension of a discrete $\Gamma \leqslant E(n)$. The first is by a direct application of Theorem \ref{thm:isometry_group_fact}.

\begin{defn}\label{def:dimension_of_group}
	Let $\Gamma \leqslant E(n)$ be discrete. The \emph{dimension} $\dim \Gamma$ of $\Gamma$ is given by $\dim \Gamma = \dim V$, where $(G, V)$ is a finite index cocompact translation pair of $\Gamma$.
\end{defn}

In order for this to be a valid definition, it needs to be independent of the choice of $(G, V)$. This is true due to the following simple lemma. 

\begin{lemma}\label{lem:dimension_uniqueness}
	Let $\Gamma$ be a discrete subgroup of $E(n)$, and suppose that $(G, V)$ and $(G', V')$ are finite index cocompact translation pairs of $\Gamma$. Then $V' = V + a$ for some $a \in \R^n$.
\end{lemma}
\begin{proof}
	Suppose first that $G = G'$. Let $g \in G$, select $x \in V$ and $y \in V'$, and denote $\tran_V(g) = v$ and $\tran_{V'}(g) = v'$. Since $g$ is an isometry, we have $\abs{x-y} = \abs{(x-y)+(v-v')}$. By similarly considering $g^{-1}$, we get $\abs{x-y} = \abs{(x-y)-(v-v')}$. Applying these identities and the parallelogram rule $\abs{a+b}^2 + \abs{a-b}^2 = 2\abs{a}^2 + 2\abs{b}^2$, we obtain
	\begin{align*}
		2\abs{x-y}^2
		&= \abs{(x-y)+(v-v')}^2 + \abs{(x-y)-(v-v')}^2\\
		&= 2\abs{x-y}^2 + 2\abs{v-v'}^2.
	\end{align*}
	Hence, $\abs{v-v'} = 0$, or in other words, $v = v'$. Since $\{\tran_{V}(g) : g \in G\}$ span a space parallel to $V$ and the same holds for $\{\tran_{V'}(g) : g \in G\}$ and $V'$, it follows that $V$ and $V'$ are parallel.
	
	Consider now the general case. Then the group $G \cap G'$ is of finite index in both $G$ and $G'$, and therefore also in $\Gamma$. If $g \in G \cap G'$, it follows that $g \vert_V$ is a translation of $V$ since $g \in G$. Hence, $(G \cap G') V \subset V$, and since $\id_{\R^n} \in G \cap G'$, we conclude that $(G \cap G') V = V$. Moreover, the space $V / (G \cap G')$ is compact; indeed, the space $V / G$ is compact, and the quotient map $p \colon V/(G \cap G') \to V / G$ is proper due to being a continuous map with finite fibres between locally compact Hausdorff spaces. 
	
	In conclusion, $(G \cap G', V)$ is a finite index cocompact translation pair of $\Gamma$. We may similarly deduce that $(G \cap G', V')$ is a finite index cocompact translation pair of $\Gamma$. Hence, by the first case, $V$ and $V'$ are parallel.
\end{proof}

The second equivalent definition is by a notion of growth for $\Gamma$. This is similar in spirit to the standard notion of growth for finitely generated discrete groups under the word metric, except we formulate the concept using Euclidean balls.

We recall for functions $f, g : [0, \infty) \to [0, \infty)$ the big $O$ and big $\Theta$ notations: we denote $f(t) = O(g(t))$ if $\limsup_{t \to \infty} f(t)/g(t) < \infty$, and $f(t) = \Theta(g(t))$ if $f(t) = O(g(t))$ and $g(t) = O(f(t))$. If $\Gamma \leqslant E(n)$ is discrete, we let $N_\Gamma(r)$ denote the number of elements $\gamma \in \Gamma$ for which $\abs{T(\gamma)} \leq r$. Furthermore, if $(G, V)$ is a cocompact translation pair of $\Gamma$, we let $N^V_G(r)$ be the number of elements $\gamma \in G$ for which $\abs{\tran_V(\gamma)} \leq r$.

The following lemma shows how the growth of $N_\Gamma(r)$ is connected to the dimension of $\Gamma$. 

\begin{lemma}\label{lem:dimension_growth_interpretation}
	Let $\Gamma$ be a discrete subgroup of $E(n)$, and let $(G, V)$ be a finite index cocompact translation pair of $\Gamma$. Then
	\[
		N_\Gamma(r) = \Theta(N_G(r)) = \Theta(N_G^V(r)) = \Theta\big(r^{\dim V}\big)
		= \Theta\big(r^{\dim \Gamma}\big).
	\]
	In particular, $\dim \Gamma = k$ if and only if $N_\Gamma(r) = \Theta(r^k)$.
\end{lemma}

Ihe proof of Lemma \ref{lem:dimension_growth_interpretation} is relatively straightforward, but we discuss it in several sub-lemmas due to its length.

\begin{lemma}\label{lem:comparability_sublemma_1}
	Let $\Gamma$ be a discrete subgroup of $E(n)$, and let $G$ be a subgroup of $\Gamma$ of finite index $m$. Then there exists a constant $C \geq 0$ for which
	\begin{equation}\label{eq:first_theta_estimate}
		N_G(r) \leq N_\Gamma(r) \leq m N_G(r+C).
	\end{equation}
\end{lemma}
\begin{proof}
	The lower bound of \eqref{eq:first_theta_estimate} is simply due to $G \subset \Gamma$. For the upper bound, we decompose $\Gamma$ into $m$ conjugacy classes $\Gamma = \gamma_1 G \cup \ldots \cup \gamma_m G$, and let $\gamma \in \Gamma$. We may then write $\gamma = \gamma_i \gamma'$, where $\gamma' \in G$ and $i \in \{1, \dots m\}$. Now $\tran(\gamma') = \ort(\gamma_i)\tran(\gamma) + \tran(\gamma_i)$, and hence we obtain for the selection $C = \max_i \{\abs{\tran(\gamma_i)}\}$ the estimate
	\[
		\abs{\tran(\gamma)} 
		= \abs{\ort(\gamma_i)\tran(\gamma)} 
		= \abs{\tran(\gamma') - \tran(\gamma_i)} 
		\leq \abs{\tran(\gamma')} + C.
	\]
	Hence, every element $\gamma \in \Gamma$ with $\abs{\tran(\gamma)} \leq r$ is a product of a $\gamma' \in G$ with $\abs{\tran(\gamma)} \leq r + C$ and one of $m$ elements $\gamma_i \in \Gamma$. The upper bound of \eqref{eq:first_theta_estimate} follows.
\end{proof}

\begin{lemma}\label{lem:comparability_sublemma_2}
	Let $\Gamma$ be a discrete subgroup of $E(n)$, and let $(G, V)$ be a finite index cocompact translation pair of $\Gamma$. Then
	\begin{equation}\label{eq:translation_part_comparison}
	\abs{\tran_V(\gamma)} \leq \abs{\tran(\gamma)} \leq \abs{\tran_V(\gamma)} + 2d(0, V),
	\end{equation}
	and therefore
	\begin{equation}\label{eq:second_theta_estimate}
		N_G^V(r - 2d(0, V)) \leq N_G(r) \leq N_G^V(r).
	\end{equation}
\end{lemma}
\begin{proof}
	Let $\gamma \in G$. Let $w$ be the orthogonal projection of 0 to $V$, in which case $w = w-0$ is orthogonal to $V$. We note that $\gamma(w) - \gamma(0) = \ort(\gamma)(w-0) = \ort(\gamma)w$. Since $\ort(\gamma)$ is conformal, and by Lemma \ref{lem:identity_restriction} $\ort(\gamma)$ is identity on the linear space parallel to $V$, we have that $\ort(\gamma)w$ is orthogonal to $V$.
	
	However, now we have
	\[
		\tran(\gamma) = \gamma(0) = \gamma(w) - \ort(\gamma)w
		= (\gamma(w) - w) + (w - \ort(\gamma)w).
	\]
	Since $w \in V$ on which $\gamma$ acts by translation, we in fact have $\gamma(w) - w = \tran_V(\gamma)$. Hence, we now have a decomposition
	\[
		\tran(\gamma) = \tran_V(\gamma) + w',
	\]
	where $\tran_V(\gamma)$ is parallel to $V$ and $w'$ is orthogonal to $V$. Moreover, $\abs{\ort(\gamma)w} = \abs{w} = d(0, V)$, and therefore $\abs{w'} \leq 2d(0, V)$. Hence, the estimate \eqref{eq:translation_part_comparison} follows, and the estimate \eqref{eq:second_theta_estimate} is an immediate consequence.
\end{proof}

\begin{proof}[Proof of Lemma \ref{lem:dimension_growth_interpretation}]
	It remains to show that $N_G^V(r) = \Theta(r^{\dim V})$; once this is done, the rest of the claim follows from Lemmas \ref{lem:comparability_sublemma_1} and \ref{lem:comparability_sublemma_2}. For this, consider the group $\tran_V(G) \leqslant \R^n$, that is, the image of $G$ under the homomorphism $\tran_V \colon G \to \R^n$. We show first that $\tran_V(G)$ is discrete. Indeed, suppose to the contrary that there exists a bounded sequence $(\tran_V(\gamma_j))_{j=1}^\infty$ of distinct elements of $\tran_V(G)$. By \eqref{eq:translation_part_comparison}, $(\tran(\gamma_j))_{j=1}^\infty$ would also be bounded, so by moving to a subsequence we may assume that it converges. However, since $O(n)$ is compact, we may also by moving to a subsequence assume that $(\ort(\gamma_j))_{j=1}^\infty$ converges. This is a contradiction, since now $(\gamma_j)_{j=1}^\infty$ is a convergent sequence of distinct elements of $G$, which contradicts the discreteness of $G$.
	
	Hence, if $V'$ is the linear space parallel to $V$, then $\tran_V(G)$ is a discrete subgroup of $V'$ and $V' / \tran_V(G)$ is compact. It follows by a classical volume counting argument that the number of elements in $\overline{B}_{\R^n}(0, r) \cap \tran_V(G) = \overline{B}_{V'}(0, r) \cap \tran_V(G)$ grows at a rate of $\Theta(r^{\dim V'}) = \Theta(r^{\dim V})$. Indeed, since $\tran_V(G)$ is a discrete subgroup of $V'$, there exists $r_0 > 0$ for which the balls $B_{V'}(x, r_0)$ for $x \in \tran_V(G)$ are disjoint. By volume considerations you can fit at most $C r^{\dim V'}$ such balls into $\overline{B}_{V'}(0, r)$ for some $C > 0$. On the other hand, since $V' / \tran_V(G)$ is compact, there exists $R_0 > 0$ for which $B_{V'}(0, R_0) / \tran_V(G) = V' / \tran_V(G)$. It follows that every point of $V'$ is contained in some $B_{V'}(x, R_0)$, $x \in \tran_V(G)$, and again due to volume you need at least $cr^{\dim V'}$ of such balls to cover $\overline{B}_{V'}(0, r)$ for some $c > 0$.
	
	Finally, let $G_0$ be the subgroup of all elements of $G$ which are identity on $V$. Then $G_0$ is finite, as otherwise there exists a sequence of distinct elements of $G_0$, and an argument similar to the one used to prove discreteness of $\tran_V(G)$ yields a contradiction with the discreteness of $G$. Let $l$ be the number of elements in $G_0$. Then for every $v \in \tran_V(G)$, the set of $\gamma \in G$ with $\tran_V(\gamma) = v$ contains exactly $l$ elements. We finally conclude that $N_G^V(r) = \Theta(lr^{\dim V}) = \Theta(r^{\dim V})$, and the claim follows.
\end{proof}

In addition to Lemma \ref{lem:dimension_growth_interpretation}, we require a growth estimate for the set of orthogonal components $\ort(\Gamma)$ that will play a key role in Section \ref{sect:length_estimate}. For a discrete $\Gamma \leqslant E(n)$, let
\[
	\ort(\Gamma) = \left\{ \ort(\gamma) : \gamma \in \Gamma \right\},
\]
and let $\Lambda_\Gamma(r)$ denote the number of different elements $\ort(\gamma) \in \ort(\Gamma)$ with $\abs{T(\gamma)} \leq r$. Moreover, similarly as before, if $(G, V)$ is a cocompact translation pair of $\Gamma$, we let $\Lambda^V_G(r)$ be the number of $\ort(\gamma) \in \ort(\Gamma)$ with $\abs{\tran_V(\gamma)} \leq r$.

\begin{lemma}\label{lem:orth_part_growth_estimate}
	Let $\Gamma$ be a discrete subgroup of $E(n)$, and let $(G, V)$ be a finite index cocompact translation pair of $\Gamma$. Let $k = \dim G$ and $l = \dim \Gamma_T$. Then 
	\[
		\Lambda_G^V(r) = \Theta(\Lambda_G(r)) = \Theta(r^{k-l})
	\]
\end{lemma}

\begin{proof}
	Suppose $\gamma_1, \gamma_2 \in G$ satisfy $\ort(\gamma_1) = \ort(\gamma_2)$ and $\tran(\gamma_1), \tran(\gamma_2) \in \overline{B^n}(r)$. Then $\ort(\gamma_1^{-1} \circ \gamma_2) = \id_{\R^n}$, and therefore $\gamma_1^{-1} \circ \gamma_2 \in \Gamma_T$. Consequently, $\gamma_2 = \gamma_1 \circ \gamma'$ for some $\gamma' \in \Gamma_T$ with $\abs{\tran(\gamma')} \leq 2r$. We obtain that 
	\[
		\Lambda_G(r) \geq \frac{N_G(r)}{N_{\Gamma_T}(2r)}.
	\]
	
	Conversely, let $\gamma_1 \in G$ with $\abs{\tran(\gamma_1)} \leq r$. We define a set
	\[
		F_{\gamma_1} = \left\{ \gamma_1 \circ \gamma' : 
			\gamma' \in \Gamma_T, \abs{\tran(\gamma')} \leq r \right\}.
	\]
	Note that $\tran(\gamma) \leq 2r$ for every $\gamma \in F_{\gamma_1}$, and that $F_{\gamma_1}$ has $N_{\Gamma_T}(r)$ different elements. Moreover, if $\gamma_2 \in G$ is another element with $\abs{\tran(\gamma_1)} \leq r$ and $\ort(\gamma_2) \neq \ort(\gamma_1)$, it follows that $F_{\gamma_1} \cap F_{\gamma_2} = \emptyset$, since every element in $F_{\gamma_i}$ has the same orthogonal part as $\gamma_i$. Hence, we obtain that $N_\Gamma(2r) \geq N_{\Gamma_T}(r) \Lambda_G(r)$, which when rearranged yields
	\[
		\Lambda_G(r) \leq \frac{N_\Gamma(2r)}{N_{\Gamma_T}(r)}.
	\]
	
	By Lemma \ref{lem:dimension_growth_interpretation}, $N_{\Gamma}(r) = \Theta(N_G(r)) = \Theta(r^k)$, and $N_{\Gamma_T}(r) = \Theta(r^l)$. Hence, $\Lambda_G(r) = \Theta(r^{k-l})$. Finally, we note that by \eqref{eq:translation_part_comparison} of Lemma \ref{lem:comparability_sublemma_2}, we obtain that
	\[
		\Lambda_G^V(r - 2d(0, V)) \leq \Lambda_G(r) \leq \Lambda_G^V(r),
	\]
	which yields the remaining claim that $\Lambda_G^V(r) = \Theta(\Lambda_G(r))$
\end{proof}

\subsection{Other properties}
We conclude this section by recalling two useful basic results related discrete subgroups of $E(n)$. The first one is a connection between dimension and finite index subgroups. 

\begin{lemma}\label{lem:dimension_finite_index_connection}
	Let $\Gamma$ be a discrete subgroup of $E(n)$, and let $G$ be a subgroup of $\Gamma$. Then $G$ is of finite index in $\Gamma$ if and only if $\dim \Gamma = \dim G$.
\end{lemma}
\begin{proof}
	Suppose first that $G$ is of finite index in $\Gamma$. By Lemma \ref{lem:dimension_growth_interpretation}, we have $N_\Gamma(r) = \Theta(r^{\dim \Gamma})$ and $N_G(r) = \Theta(r^{\dim G})$. Hence, it follows from Lemma \ref{lem:comparability_sublemma_1} that $r^{\dim \Gamma} = \Theta(r^{\dim G})$, which in turn implies $\dim \Gamma = \dim G$.
	
	Suppose then that $\dim \Gamma = \dim G = k$. Then by Lemma \ref{lem:dimension_growth_interpretation}, there exist $C_1, C_2, r_0 > 0$ such that $N_\Gamma(r)/r^k \leq C_1$ and $r^k/N_G(r) \leq C_2$ when $r > r_0$. Suppose to the contrary that $G$ has infinitely many conjugacy classes $\gamma_j G$. For every positive integer $i$, let $r_i$ be such that $\abs{T(\gamma_j)} \leq r_i$ for $i$ different $\gamma_j$.
	Now, if $\gamma \in G$ is such that $\abs{\tran(\gamma)} \leq r - r_i$ and $\gamma_j$ is such that $\abs{\tran{\gamma_j}} \leq r_i$, then $\abs{\tran(\gamma_j \gamma)} \leq r$. Hence, for $\gamma_j$ as above, the set of elements $\gamma' \in \gamma_j G$ with $\abs{\tran(\gamma')} \leq r$ has at least $N_G(r-r_i)$ elements. Therefore, we have obtained $N_\Gamma(r) \geq i N_G(r-r_i)$ for every $i > 0$ and $r > r_i$. 
	
	Suppose then that $i > 0$, $r > r_0$ and $r > 2r_i$. Now we have
	\[
	C_1 r^k \geq N_\Gamma(r) \geq i N_G(r-r_i) \geq  \frac{i(r-r_i)^k}{C_2} >  \frac{ir^k}{2^k C_2}.
	\]
	This is clearly a contradiction if $i > C_1 C_2 2^k$, completing the proof.
\end{proof}

Finally, we note that for groups of dimension at least $n-1$, we in fact have a stronger generalization of Bieberbach's first theorem than the one given by Theorem \ref{thm:isometry_group_fact}. The example given in Section \ref{subsect:translation_subgroup} shows that $n-1$ is the lowest dimension where this is possible.

\begin{lemma}\label{lem:translation_subgroup_special_case}
	Let $\Gamma$ be a discrete subgroup of $E(n)$, and suppose $\dim \Gamma \geq n-1$. Then $\dim \Gamma_T = \dim \Gamma$.
\end{lemma}
\begin{proof}
	Case $\dim \Gamma = n$ is precisely Bieberbach's first theorem. Suppose then that $\dim \Gamma = n-1$. Let $(G, V)$ be a finite index cocompact translation pair of $\Gamma$, and let $W$ be the linear space orthogonal to $V$. Let $\gamma \in G$ and suppose $\gamma \vert_V$ is a translation by $a \neq 0$. Then $\gamma$ is completely determined by $a$ and $\gamma \vert_W$. Since $W$ is 1-dimensional, we have 2 options for $\gamma \vert_W$: either it is identity, in which case $\gamma$ is a translation, or it reflects $W$ across $V$, in which case $\gamma \circ \gamma$ is a translation with translation vector $2a$. The claim follows.
\end{proof}


\section{Discrete isometry groups and linear conformal maps}\label{sect:lin_conf_interaction}

Suppose that $A \colon \R^n \to \R^n$ is a linear conformal map, and that $\Gamma$ is a discrete subgroup of $E(n)$. Then $A$ is of the form $A = \lambda A'$, where $A' \in O(n)$ and $\lambda > 0$. Moreover, since we require conformal maps to preserve orientation, $A'$ is in fact orientation preserving. A necessary condition for $A$ and $\Gamma$ to be part of a Latt\`es triple $(\Gamma, h, A)$ into a Riemannian manifold $M$ is that $A\Gamma A^{-1} \subset \Gamma$. In this section, we look into the implications of this condition.

We begin by considering the most restrictive case of a non-expanding $A$ satisfying $A\Gamma A^{-1} \subset \Gamma$. For $B = cB'$ with $c > 0$ and $B' \in O(n)$, we denote by $\theta_B \colon E(n) \to E(n)$ the conjugation homomorphism given by $\theta_B(\gamma) = B \gamma B^{-1}$ for $\gamma \in E(n)$. Note that $\theta_B$ is an ismomorphism with inverse $\theta_{B^{-1}}$.

\begin{lemma}\label{lem:Gamma_A_interaction_nonexpanding}
	Let $\Gamma$ be a discrete subgroup of $E(n)$, and let $A = \lambda A'$, where $A' \in O(n)$ and $\lambda > 0$. Suppose that $A \Gamma A^{-1} \subset \Gamma$ and $\lambda \leq 1$. Then $A \Gamma A^{-1} = \Gamma$.
\end{lemma}

\begin{proof}
	Let $\gamma \in \Gamma$. The goal is to show that $\gamma = \theta_{A} (\gamma')$ for some $\gamma' \in \Gamma$. Our strategy is to show that $\theta_{A^k} (\gamma) = \gamma$ for some $k > 0$. Then, since $\theta_{A^l} (\gamma) \in \Gamma$ for every $l \geq 0$, we may select $\gamma' = \theta_{A^{k-1}} (\gamma)$.
	
	Let $r = \abs{\tran(\gamma)}$, and let $\Gamma_r = \{\gamma' \in \Gamma : \abs{\tran(\gamma')} \leq r \}$. Since the map given by $\gamma' \mapsto \abs{\tran(\gamma')}$ is continuous, $\Gamma_r$ is a closed discrete subset of the compact space $O(n) \times \overline{B^n}(0, r) \subset E(n)$. Therefore, $\Gamma_r$ is finite. Since $\theta_{A}\vert_{\Gamma_r}$ is an injective self-map on a finite set, $\theta_{A} \vert_{\Gamma_r}$ is a permutation on $\Gamma_r$, and therefore there exists $k > 0$ for which $\theta_{A}^k$ is the identity on $\Gamma_r$. Hence, $\theta_{A^k}(\gamma) = \theta_{A}^k(\gamma) = \gamma$, and the proof is complete.
\end{proof}

We remark that the case $\lambda < 1$ in Lemma \ref{lem:Gamma_A_interaction_nonexpanding} is in fact even more restricted. In particular, if $\lambda < 1$, then for any $k > 0$ the identity $\theta_{A^k}(\gamma) = \gamma$ implies $\tran(\gamma) = 0$. Hence, $\Gamma$ is in fact contained in $O(n)$, and since $\Gamma$ is also closed and discrete, it is finite.

Next, we observe the general case where $A$ may be expanding. While the condition $A \Gamma A^{-1} \subset \Gamma$ does not imply $A \Gamma A^{-1} = \Gamma$ for expanding $A$, we still have the following.

\begin{lemma}\label{lem:Gamma_A_interaction_index}
	Let $\Gamma$ be a discrete subgroup of $E(n)$, and let $A = \lambda A'$, where $A' \in O(n)$ and $\lambda > 0$. Suppose that $A \Gamma A^{-1} \subset \Gamma$. Then $A \Gamma A^{-1}$ is a subgroup of finite index in $\Gamma$.
\end{lemma}
\begin{proof}
	Clearly $A \Gamma A^{-1}$ is a subgroup. To show that $A \Gamma A^{-1}$ has finite index in $\Gamma$, the strategy is to show that $\dim A \Gamma A^{-1} = \dim \Gamma$, after which the result follows from Lemma \ref{lem:dimension_finite_index_connection}. 
	
	We note that the restriction $\theta_A\vert_{\Gamma}$ is a bijection $\Gamma \to A \gamma A^{-1}$. Moreover, if $\gamma \in \Gamma$, then $\tran(\theta_A(\gamma)) = A\tran(\gamma)$ and consequently $\abs{\tran(\theta_A(\gamma))} = \lambda \abs{\tran(\gamma)}$. Hence, we obtain that
	\[
		N_\Gamma(r) = N_{A \Gamma A^{-1}}(\lambda r). 
	\]
	Therefore, by Lemma \ref{lem:dimension_growth_interpretation}, $A \Gamma A^{-1}$ and $\Gamma$ have the same dimension.
\end{proof}

Next, we consider the interaction of $A$ with a finite index cocompact translation pair $(G, V)$ of $\Gamma$. We may apply Lemma \ref{lem:Gamma_A_interaction_index} to obtain the following result.

\begin{lemma}\label{lem:conjugation_map_parallel_space}
	Let $\Gamma$ be a discrete subgroup of $E(n)$, let $(G, V)$ be a finite index cocompact translation pair of $\Gamma$, and let $A = \lambda A'$ where $A' \in O(n)$ and $\lambda > 0$. Suppose that $A \Gamma A^{-1} \subset \Gamma$. Then $AV = V + a$ for some $a \in \R^n$.
\end{lemma}
\begin{proof}
	 We show that $(AGA^{-1}, AV)$ is a finite index cocompact translation pair of $\Gamma$, after which the claim follows from Lemma \ref{lem:dimension_uniqueness}. Since $G$ acts on $V$ by translations and $V / G$ is compact, the group $AGA^{-1}$ acts on $AV$ by translations and $AV / AGA^{-1}$ is compact. It remains to verify that $AGA^{-1}$ is of finite index in $\Gamma$. This follows since Lemma \ref{lem:Gamma_A_interaction_index} yields that $AGA^{-1}$ is of finite index in $G$.
\end{proof}

Next, we show that if the map $A$ is expanding, then we may in fact find a finite index cocompact translation pair $(G, V)$ for which $AV = V$.

\begin{lemma}\label{lem:conjugation_map_linearization}
	Let $\Gamma$ be a discrete subgroup of $E(n)$, and let $A = \lambda A'$, where $A' \in O(n)$ and $\lambda > 1$. The there exists a finite index cocompact translation pair $(G, V)$ of $\Gamma$ for which $G$ is abelian, $V$ is linear, $AV = V$, and $G$ contains a finite index subgroup of $\Gamma_T$.
\end{lemma}
\begin{proof}
	\enlargethispage{\baselineskip}
	By Theorem \ref{thm:isometry_group_fact}, there exists a finite index cocompact translation pair $(G, V)$ of $\Gamma$ for which $G$ is abelian and $\Gamma_T \subset G$. We note that by Lemma \ref{lem:conjugation_map_parallel_space}, the spaces $A^m V \subset \R^n$ are parallel for $m \in \N$. Let $W \subset \R^n$ be the orthogonal complement of $V$ which contains $0$, and let $v_0$ be the unique point of intersection of $V$ and $W$. Since the spaces $A^m V$ are parallel to $V$, $W$ also intersects each space $A^m V$ at a unique point $v_m$. See Figure \ref{fig:affine_selection} for an illustration of the selection of $v_i$.
	
	\begin{figure}[h]
		\centering
		\begin{tikzpicture}
			\fill (0,0) circle[radius=1.3pt] node[anchor=east]{$0$};
			
			\draw (-2, -.5) -- (0,.5) -- (2,1.5) node[anchor=west]{$V$};
			\draw (-2, -2) -- (0,-1) -- (2,-0) node[anchor=west]{$AV$};
			\draw (-2, .5) -- (0,1.5) -- (2,2.5) node[anchor=west]{$A^2V$};
			
			\draw (1, -2) -- (0,0) -- (-1.25,2.5) node[anchor=west]{$W$};
			\draw (-2.5, -1.25) -- (0,0) -- (3.5,1.75) node[anchor=west]{$W^\perp$};
			
			\fill (-1/5,2/5) circle[radius=1.3pt] node[anchor=east]{$v_0$};
			\fill (2/5,-4/5) circle[radius=1.3pt] node[anchor=east]{$v_1$};
			\fill (-3/5,6/5) circle[radius=1.3pt] node[anchor=east]{$v_2$};
			
			\draw (2/15,1/15) -- (1/15,3/15) 
				-- (-1/15, 2/15);
			\draw (-1/5 + 2/15,2/5 + 1/15) -- (-1/5 + 1/15, 2/5 + 3/15) 
				-- (-1/5 - 1/15, 2/5 + 2/15);
			\draw (2/5 + 2/15,-4/5 + 1/15) -- (2/5 + 1/15, -4/5 + 3/15) 
				-- (2/5 - 1/15, -4/5 + 2/15);
			\draw (-3/5 + 2/15, 6/5 + 1/15) -- (-3/5 + 1/15, 6/5 + 3/15) 
				-- (-3/5 - 1/15, 6/5 + 2/15);
		\end{tikzpicture}
		\caption{An illustration of the selection process of the points $v_m$. Note that despite being represented by lines in the picture, the spaces depicted may be higher dimensional.}
		\label{fig:affine_selection}
	\end{figure}
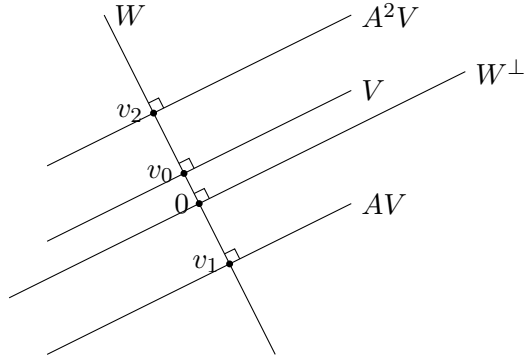

	We note that if $W^\perp$ is the linear orthogonal complement of $W$, then $W^\perp$ is parallel to $V$. Hence, linearity of $A$ and the fact that $AV$ is parallel to $V$ yields that $AW^\perp = W^\perp$, and thus conformality of $A$ yields $AW = W$. It follows that $v_m = A^m(v_0)$ for every $m \geq 0$. 
	
	Next, we show that $0$ is an affine combination of the elements $v_m$. Suppose to the contrary that $0 \notin W' = \aff(\{v_0, v_1, \ldots\})$, where $\aff(S)$ denotes the affine hull of a set $S \subset \R^n$, defined by
	\[
		\aff(S) = \left\{
			\sum_{i=1}^m a_i s_i : m \in \Z_+, s_i \in S, a_i \in \R, \sum_{i=1}^m a_i = 1
		\right\}.
	\]  
	Let $L$ be the line passing through $0$ and $v_0$. Then $L$ intersects $W'$ at $v_0 \neq 0$, and $L$ is not contained in $W'$ since $0 \notin W'$. It follows that $W'$ contains no 1-dimensional affine subspace parallel to $L$. 
	
	Let $0 < c < 1$, and consider a cone $C = \{x \in \R^n : d(x, L) \leq c\abs{x}\}$ around $L$. Since $W'$ has no 1-dimensional affine subspaces parallel to $L$, the intersection $W' \cap C$ is bounded if $c$ is sufficiently small. However, by using Lemma \ref{lem:basic_prop_of_O(n)} on $A'$, we find arbitrarily large $m \in \N$ for which $A^m L \subset C$. It follows that for such $m$ we have $v_m = A^m v_0 \in C \cap W'$. On the other hand, since $A$ is expanding and $v_0 \neq 0$, we have that $\abs{A^m(v_0)}$ becomes arbitrarily large as $m$ increases. Since $W' \cap C$ is bounded, this is a contradiction. See Figure \ref{fig:affine_contradiction} for an illustration.
	
	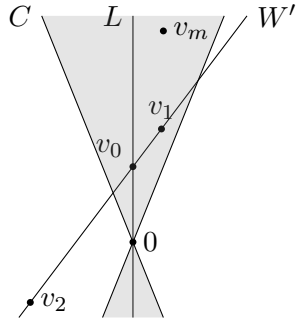
\begin{figure}[h]
		\centering
		\begin{tikzpicture}
			\fill (0,0) circle[radius=1.3pt] node[anchor=west]{$0$};
			\draw (0,-1) -- (0, 0) -- (0,3) node[anchor=east]{$L$};
			\draw (-1.5, -1) -- (0,1) -- (1.5,3) node[anchor=west]{$W'$};
			
			\fill (0,1) circle[radius=1.3pt] node[anchor=south east]{$v_0$};
			\fill (1.5/4,1.5) circle[radius=1.3pt] node[anchor=south]{$v_1$};
			\fill (-1.35,-0.8) circle[radius=1.3pt] node[anchor=west]{$v_2$};
			
			\draw [draw=black, fill=gray, fill opacity=0.2]
				(1.2,3) -- (0,0) -- (-1.2,3);
			\draw (-1.2,3) node[anchor=east]{$C$};
			\draw [draw=black, fill=gray, fill opacity=0.2]
				(0.4,-1) -- (0,0) -- (-0.4,-1);
			
			\fill (.4,2.8) circle[radius=1.3pt] node[anchor=west]{$v_m$};
		\end{tikzpicture}
		\caption{The contradiction which shows that $0 \in W' = \aff(\{v_0, v_1, \ldots\})$: by selecting exponents $m$ where the orthogonal component of $A^m$ is close to $\id_{\R^n}$, we find points $v_m = A^m(v_0)$ in the cone $C$ arbitrarily far from $0$.}
		\label{fig:affine_contradiction}
	\end{figure}
	
	Hence, we have that $a_0 v_0 + \cdots + a_m v_m = 0$ for some $m \geq 0$ and $a_i \in \R$ with $a_0 + \cdots + a_m = 1$. Let $V' = a_0 V + \cdots + a_m A^m V$. Then $V'$ is an affine space parallel to $V$, and moreover since $0 \in V'$, the space $V'$ is in fact linear. 
	
	Let $G' = G \cap AGA^{-1} \cap \ldots \cap A^m G A^{-m}$. Then $G'$ is a subgroup of finite index in $\Gamma$, and as in the proof of Lemma \ref{lem:dimension_uniqueness} we have for all $g \in G'$ and $0 \leq k \leq m$ that $g\vert_{A^k V}$ is a translation with translation vector $\tran_V(g)$. Since the maps $g \in G'$ are affine, we also have for every $g \in G'$ that $g\vert_{V'}$ is a translation and $\tran_{V'}(g) = \tran_V(g)$. Moreover, since $G' \leqslant G$ and $G$ is abelian, $G'$ is also abelian. 
	
	Finally, we note that for every $k \in \{1, \ldots, \infty\}$ we have that $A^k\Gamma_TA^{-k}$ is contained in $\Gamma_T$, and hence by Lemma \ref{lem:Gamma_A_interaction_index} $A^k\Gamma_TA^{-k}$ is a finite index subgroup in $\Gamma_T$. Since $\Gamma_T \subset G$, we therefore have that $\bigcap_{k=1}^m A^k\Gamma_TA^{-k}$ is a finite index subgroup of $\Gamma_T$ contained in $G'$. We conclude that $(G', V')$ is the desired finite index cocompact translation pair of $\Gamma$.
\end{proof}

Before beginning the proof of Theorem \ref{prop:translations_under_A-condition}, we require the following standard lemma about the commutation of elements of $O(n)$ proven in eg.\ \cite[Lemma 3.2.4]{Wolf1967book} and \cite[Lemma 2.1]{Szczepanski2012book}. For any two elements $h_1, h_2$ in a group $H$, we recall the \emph{commutator} $[h_1, h_2] = h_1 h_2 h_1^{-1} h_2^{-1}$.

\begin{lemma}\label{lem:wolf_conjugation_lemma}
	For each $n \in \Z_+$, there is a neighborhood $U_n \subset O(n)$ of $\id_{\R^n} \in O(n)$ with the following property: if $A, B \in U_n$ and $A$ commutes with $[A, B]$, then $A$ commutes with $B$.
\end{lemma}

With this, we may finally prove Theorem \ref{prop:translations_under_A-condition}, which we recall is stated as follows.

\begin{customthm}{\ref{prop:translations_under_A-condition}}
	Let $\Gamma \leqslant E(n)$ be discrete, and let $A \colon \R^n \to \R^n$ be a linear conformal map. Suppose that $A$ is expanding and $A \Gamma A^{-1} \subset \Gamma$. Then $\dim \Gamma_T = \dim \Gamma$.
\end{customthm}

\begin{proof}
	Denote $A = \lambda A'$, with $\lambda > 1$ and $A' \in O(n)$. By Lemma \ref{lem:basic_prop_of_O(n)}, there is a subsequence of $((A')^m)_{m=1}^\infty$ converging to $\id_{\R^n}$. Hence, there is a $m > 0$ for which $(A')^m \in U_n$, where $U_n$ is given by Lemma \ref{lem:wolf_conjugation_lemma}. By replacing $A$ with $A^m$, we may assume that $m = 1$.
	
	By Lemma \ref{lem:conjugation_map_linearization}, we may select a cocompact translation pair $(G, V)$ of $\Gamma$ for which $V$ is linear, $G$ is abelian and $AV = V$, and $G$ contains a finite index subgroup $G'$ of $\Gamma_T$. Let $V_T = \Span \left\{\tran (\gamma) : \gamma \in \Gamma_T \right\}$. Since $G'$ is of finite index in $\Gamma_T$, we also have $V_T = \Span \left\{\tran (\gamma) : \gamma \in G' \right\}$. Since $V = \Span \left\{\tran_V (\gamma) : \gamma \in G \right\}$ and $\tran_V(g) = \tran(g)$ for all $g \in G'$, we have $V_T \subset V$. 
	
	Moreover, since $A$ is linear, it follows that $A V_T = \Span \left\{A(\tran (\gamma)) : \gamma \in \Gamma_T \right\}$. However, if $\gamma \in \Gamma_T$, then $A \gamma A^{-1} \in \Gamma_T$ and $\tran(A \gamma A^{-1}) = A(\tran(\gamma))$. It follows that $AV_T \subset V_T$, and therefore $AV_T = V_T$ since $A$ is dimension-preserving. We further denote $W = V_T^\perp \cap V$, and note that due to conformality of $A$, also $AW = W$.
	
	Suppose then towards contradiction that we have $\dim \Gamma_T < \dim \Gamma$. Then there exists $g \in G$ with $\tran_V(g) \notin V_T$, and therefore the $W$-component of $\tran_V(g)$ is nonzero. Note that $\tran(g) = \tran_V(g) + v'$, where $v' \in \R^n$ is orthogonal to $V$. Since $W \subset V$, we therefore also have that the $W$-component of $\tran(g)$ is nonzero. 
	
	We show that there is an $m \in \Z_+$ for which $A g^m A^{-1} \in G$. Indeed, if we denote $G_i = A g^i A^{-1} G$, then $G_i$ are left cosets of $G$ in $\Gamma$. Since $G$ is of finite index in $\Gamma$, it has only finitely many different left cosets in $\Gamma$, and therefore we find $i, j \in \Z_+$ for which $i > j$ and $G_i = G_j$. Now a selection of $m = i-j$ yields the desired $A g^m A^{-1} \in G$.
	
	It follows now that $A g^{lm} A^{-1} \in G$ for all $l \in \Z$. Using Lemma \ref{lem:basic_prop_of_O(n)} on the sequence $(\ort(g^{-im}))_{i=1}^\infty$, we may fix $l > 0$ satisfying $\ort(g^{-lm}) \in U_n$. Now both $g^{-lm}$ and $[g^{-lm}, A]$ are elements of $G$. Since $G$ is abelian, $g^{-lm}$ and $[g^{-lm}, A]$ commute. Thus, Lemma \ref{lem:wolf_conjugation_lemma} yields that $\ort(g^{-lm})$ and $A'$ commute, and therefore $\ort([g^{-lm}, A]) = \id_{\R^n}$. Hence, the element $[g^{-lm}, A] \in G$ is a translation.
	
	However, we have $\tran([g^{-lm}, A]) = lm(A-1)\tran(g)$. Since $A$ is expanding, $AW = W$ and the $W$-component of $\tran(g)$ is nonzero, we have that the $W$-component of $\tran([g^{-lm}, A])$ is nonzero. But now $[g^{-lm}, A] \in G$ is a translation with a translation vector not contained in $V_T$. This is a contradiction, which completes the proof.
\end{proof}


\section{Extension to omitted points}\label{sect:extension}

In this section, we give a simple proof for the following lemma.
\begin{lemma}\label{lem:extension_and_non-injectivity}
	Let $(\Gamma, h, A)$ be a Latt\`es triple into a closed, connected, oriented Riemannian $n$-manifold $M$, and let $g \colon h(\R^n) \to h(\R^n)$ be the induced Latt\`es map. Then $g$ extends to a quasiregular map $g \colon M \to M$. Moreover, if $\dim \Gamma > 0$, then $g$ is non-injective if and only if $A$ is expanding.
\end{lemma}

The key lemma used to prove Lemma \ref{lem:extension_and_non-injectivity} is the following.
\begin{lemma}\label{lem:multiplicity_lemma}
	Let $(\Gamma, h, A)$ be a Latt\`es triple into a closed, connected, oriented Riemannian $n$-manifold $M$, and let $g \colon h(\R^n) \to h(\R^n)$ be the induced Latt\`es map. Then the multiplicity of $g$ is equal to the index $[\Gamma : A \Gamma A^{-1}]$.
\end{lemma}
\begin{proof}[Proof of Lemma \ref{lem:extension_and_non-injectivity} assuming Lemma \ref{lem:multiplicity_lemma}]
	By the Holopainen--Rickman Picard theorem \cite{HolopainenRickman1992paper}, $M \setminus h(\R^n)$ is finite. Since $g$ has finite multiplicity by Lemma \ref{lem:multiplicity_lemma}, we have for every $y \in M$ that $g^{-1}\{y\}$ is finite and therefore does not have any accumulation points. Hence, the extension part is an immediate corollary of eg.\ \cite[Theorem 2.6]{OkuyamaPankka2013paper}.
	
	Suppose $\dim \Gamma > 0$. If $A$ is not expanding, we have by Lemmas \ref{lem:Gamma_A_interaction_nonexpanding} and \ref{lem:multiplicity_lemma} that $g \colon h(\R^n) \to h(\R^n)$ is injective. This conclusion extends to $g \colon M \to M$ since $g \colon M \to M$ has at most $\deg g$ preimages at every $y \in M$, with equality for almost every $y \in M$.
	
	It remains to show non-injectivity of $g$ if $A$ is expanding. For this, we may use discreteness of $\Gamma$ and the fact that $\dim \Gamma > 0$ to select a $\gamma_0 \in \Gamma$ with minimal positive $\abs{\tran(\gamma_0)}$. Now, since $\tran(A\gamma A^{-1}) = A\tran(\gamma)$ for all $\gamma \in \Gamma$, expandingness of $A$ implies that $\gamma_0 \notin A\Gamma A^{-1}$, and therefore that $[\Gamma : A \Gamma A^{-1}] > 1$. The result then follows from Lemma \ref{lem:multiplicity_lemma}.
\end{proof}

It remains to prove Lemma \ref{lem:multiplicity_lemma}. For this, we require two easy lemmas on discrete subgroups of $E(n)$.

\begin{lemma}\label{lem:index_correspondence}
	Let $\Gamma$ be a discrete subgroup of $E(n)$, and let $A = \lambda A'$, where $A' \in O(n)$ and $\lambda > 0$. Suppose that $A \Gamma A^{-1} \subset \Gamma$. Then $A^{-1} \Gamma A \subset E(n)$ is a group, $\Gamma$ is a subgroup of $A^{-1} \Gamma A$, and 
	\[
		[A^{-1} \Gamma A : \Gamma] = [\Gamma : A \Gamma A^{-1}].
	\]
\end{lemma}
\begin{proof}
	Verifying that $A^{-1} \Gamma A$ is a subgroup of $E(n)$ is trivial. If $\gamma \in \Gamma$, we have $\gamma = A^{-1} (A \gamma A^{-1}) A \in A^{-1} \Gamma A$, and therefore $\Gamma \leqslant A^{-1} \Gamma A$. Finally, if $\gamma_1, \gamma_2 \in \Gamma$, then it is easily seen that $\gamma_1 (A \Gamma A^{-1}) = \gamma_2 (A \Gamma A^{-1})$ if and only if $A^{-1} \gamma_1 A (\Gamma) = A^{-1} \gamma_2 A (\Gamma)$, which in turn shows the desired $[A^{-1} \Gamma A : \Gamma] = [\Gamma : A \Gamma A^{-1}]$.
\end{proof}
\begin{lemma}\label{lem:fixpoint_lemma}
	Let $\Gamma$ be a discrete subgroup of $E(n)$. Then for (Lebesgue) almost every $x \in \R^n$, the only element of $\Gamma$ which fixes $x$ is $\id_{\R^n}$.
\end{lemma}
\begin{proof}
	It is easily seen that for every $\gamma \in E(n)$, the set of fixed points of $\gamma$ is an affine subspace of $\R^n$. If $\gamma \neq \id_{\R^n}$, then this subspace of fixed points is of dimension less than $n$, and therefore of measure zero. The claim now follows since $\Gamma$ is countable, see eg.\ Lemma \ref{lem:dimension_growth_interpretation}.
\end{proof}

We now proceed with the proof of Lemma \ref{lem:multiplicity_lemma}. 

\begin{proof}[Proof of Lemma \ref{lem:multiplicity_lemma}]
	Suppose that $x, x' \in \R^n$. We see that $g(h(x)) = g(h(x'))$ if and only if $h(A(x)) = h(A(x'))$, which in turn by strong automorphicness is true if and only if $x' \in A^{-1}\Gamma A x$. Therefore, the number of elements in $g^{-1} \{g(h(x))\}$ is equal to the number of elements in $h(A^{-1} \Gamma A x)$.
	
	By automorphicness, $h$ maps the image of $x$ under every right coset of $\Gamma$ in $A^{-1} \Gamma A$ to a single point. Therefore, $h(A^{-1} \Gamma A x)$ has at most $[A^{-1} \Gamma A : \Gamma]$ many elements. Since $[A^{-1} \Gamma A : \Gamma] = [\Gamma : A \Gamma A^{-1}]$ by Lemma \ref{lem:index_correspondence}, it follows that the multiplicity of $g$ is at most $[\Gamma : A \Gamma A^{-1}]$.
	
	For the converse estimate, let $x \in \R^n$ be such that $A(x)$ is not a fixed point of any $\gamma \in \Gamma\setminus\{\id_{\R^n}\}$: this holds for almost every $x \in \R^n$ by Lemma \ref{lem:fixpoint_lemma}. Suppose that $h(A^{-1} \gamma_1 A x) = h(A^{-1} \gamma_2 A x)$, where $\gamma_1, \gamma_2 \in \Gamma$. Then by strong automorphicness, $A^{-1} \gamma_1 A x = \gamma A^{-1} \gamma_2 A x$ for some $\gamma \in \Gamma$. 
	
	It follows that $\gamma_1^{-1} (A \gamma A^{-1}) \gamma_2$ is an element of $\Gamma$ which fixes $A(x)$, and therefore $\gamma_1^{-1} (A \gamma A^{-1}) \gamma_2 = \id_{\R^n}$. From this, it follows that $\Gamma A^{-1} \gamma_1 A = \Gamma A^{-1} \gamma_2 A$. Therefore, $h$ maps the image of $x$ under every right coset of $\Gamma$ in $A^{-1} \Gamma A$ to a unique point. In conclusion, there exist points $x \in \R^n$ for which $h(A^{-1} \Gamma A x)$ has $[\Gamma : A \Gamma A^{-1}]$ many points, proving the other required estimate that the multiplicity of $g$ is at least $[\Gamma : A \Gamma A^{-1}]$.
\end{proof}


\section{Proof of Theorem \ref{prop:surjectivity_based_on_cohomology}}\label{sect:extra_proposition}

In this section, we prove Theorem \ref{prop:surjectivity_based_on_cohomology}, which is an application of \cite[Theorem 1.2]{Kangasniemi2017paper} on the question of which closed manifolds admit Latt\`es-type uniformly quasiregular maps. Recall the statement of \cite[Theorem 1.2]{Kangasniemi2017paper} that if $f \colon M \to M$ is a non-constant non-injective uniformly quasiregular map on a closed connected oriented Riemannian manifold $M$, and $M$ is not a rational cohomology sphere, then the Julia set of $f$ has positive measure. 

The essential idea behind Theorem \ref{prop:surjectivity_based_on_cohomology} is as follows: if $g$ is a Latt\`es map induced by $(\Gamma, h, A)$, $A$ is expanding and $V$ is the linear space of Lemma \ref{lem:conjugation_map_linearization}, then the part of the Julia set of $g$ contained in $h(\R^n)$ is also contained in $h(V)$. Hence, if $V$ were of a dimension other than $n$, it would imply that $g$ has a Julia set of zero measure, which is prevented by \cite[Theorem 1.2]{Kangasniemi2017paper}.

We now give the detailed proof, recalling first the statement of the Proposition.

\begin{customthm}{\ref{prop:surjectivity_based_on_cohomology}}
	Let $M$ be a closed, connected, and oriented Riemannian $n$-manifold, and let $(\Gamma, h, A)$ be a Latt\`es triple into $M$. Suppose that $A$ is expanding and $M$ is not a rational cohomology sphere. Then $\dim \Gamma = n$ and $h$ is surjective.
\end{customthm}

\begin{proof}
	Since $A$ is expanding, we may select $(G, V)$ as in Lemma \ref{lem:conjugation_map_linearization}. Since $GV = V$ and $G$ is a finite index subgroup of $\Gamma$, we have $\Gamma V = V_1 \cup \ldots \cup V_k$, where $V_i \subset \R^n$ are affine subspaces of $\R^n$. Moreover, since $(\gamma G \gamma^{-1}, \gamma V)$ is a finite index cocompact translation pair of $\Gamma$ for every $\gamma \in \Gamma$, we see that the spaces $V_i$ are parallel to $V$. We denote $V' = \Gamma V$, and note that $\Gamma V' = V'$.
	
	Suppose towards contradiction that $\dim \Gamma < n$. It is enough to show that $h(\R^n \setminus V)$ is contained in the Fatou set $\Fatou_g$ of $g$. Indeed, assuming that this is shown, we know that the Julia set $\Julia_g$ of $g$ is contained within $h(V) \cup (M \setminus h(\R^n))$. Since $h$ is quasiregular, $h(V)$ has Lebesgue measure zero, see e.g.\ Rickman \cite[I.4.14]{Rickman1993book}. Therefore, since $M \setminus h(\R^n)$ is finite, $\Julia_g$ has Lebesgue measure zero. Since $g$ is non-injective by Lemma \ref{lem:extension_and_non-injectivity}, this contradicts \cite[Theorem 1.2]{Kangasniemi2017paper}.
	
	Hence, let $x \in \R^n \setminus V$. We show that $h(x) \in \Fatou_g$. Let $\lambda > 1$ be the expansion factor of $A$, let $l = d(x, V)$, and let $r > 0$ be such that $r < l$. Since $A$ is a linear conformal map and $AV = V$, we have for every $m \in \Z_+$ that $d(A^m B^n(x, r), V) = \lambda^m(l-r)$. Hence, there exists $m_0 \in \Z_+$ for which $d(A^m B^n(x, r), V) > 1 + \sup_i d(V_i, V)$ whenever $m \geq m_0$. It follows that $A^m B^n(x,r) \cap B^n(V', 1) = \emptyset$ for all $m \geq m_0$.
	
	We denote $U_x = h(B^n(x, r))$ and $U_V = h(B^n(V', 1))$, and note that $U_x$ and $U_V$ are open. Now, suppose that $g^{m}(U_x) \cap U_V \neq \emptyset$ for some $m \geq m_0$. Since $g \circ h = h \circ A$, we obtain that $h(A^{m}B^n(x, r)) \cap h(B^n(V', 1)) \neq \emptyset$. By the strong automorphicness of $h$, we have $A^{m}B^n(x, r) \cap \Gamma B^n(V', 1) \neq \emptyset$. However, since $\Gamma V' = V'$ and $\Gamma$ is a group of isometries, it follows that $\Gamma B^n(V', 1) = B^n(V', 1)$. Thus, $A^{m}B^n(x, r) \cap B^n(V', 1) \neq \emptyset$, which is a contradiction. Hence, we obtain that $g^{m}(U_x) \cap U_V = \emptyset$ for all $m \geq m_0$. 
	
	Therefore, the family $\left\{g^{m} \vert U_x : m \geq m_0 \right\}$ consists of $K$-quasiregular mappings which omit $U_V$, which implies that the family is normal; see e.g.\ \cite[Proposition 6.1]{OkuyamaPankka2014paper}. Since $U_x$ is a neghborhood of $h(x)$, we obtain the desired result $h(x) \in \Fatou_f$, concluding the proof.
\end{proof}


\section{Limits of automorphic quasiregular maps}\label{sect:limit_lemma}

The rest of this paper is dedicated to the proof of Theorem \ref{thm:generalized_martio_periodicity_result}. In this section, we generalize a lemma of Martio \cite[Lemma 3.1]{Martio1975paper_kperiod} to the automorphic case. We obtain the remaining part of Theorem \ref{prop:extension_result} on the number of omitted points as a byproduct. 

For the purposes of the following discussion, let $M$ be a closed Riemannian manifold, let $\Gamma \leqslant E(n)$ be discrete, let $f \colon \R^n \to M$ be automorphic with respect to $\Gamma$, and let $(G,V)$ be a finite index cocompact translation pair of $\Gamma$. 

We begin by discussing coordinate changes in our setting. Given an isometry $L \in E(n)$, we define a coordinate change by $L$ for $\Gamma$, $G$, $V$ and $f$ by setting $\Gamma' = L \Gamma L^{-1}$, $G' = L G L^{-1}$, $V' = LV$ and $f' = f \circ L^{-1}$. The resulting $\Gamma'$ is a discrete subgroup of $E(n)$, the map $f'$ is automorphic with respect to $\Gamma'$, and $(G', V')$ is a finite index cocompact translation pair of $\Gamma'$. Many other properties of $\Gamma$, $f$, $G$ and $V$ are also preserved, such as the dimensions of $\Gamma$, $G$ and $V$, and whether $G$ is a normal abelian subgroup of $\Gamma$. We note, however, that an extra assumption of $L$ preserving orientation is required in order to preserve quasiregularity of $f$ in the coordinate change.

Next, we generalize the concept of a period strip to our setting. We consider first the standard definition for a $k$-periodic quasiregular map $f \colon \R^n \to M$. Recall that an element $v \in \R^n\setminus \{0\}$ is a \emph{period} of $f$ if $f(x+v) = f(x)$ for every $x \in \R^n$, and $f$ is \emph{$k$-periodic} if the periods of $f$ span a $k$-dimensional subspace. If $f$ is $k$-periodic, then there exists a free generating set $\{v_1, \ldots, v_k\}$ for the periods of $f$, since the discreteness of $f$ implies that the periods of $f$ along with 0 form a discrete subgroup of $\R^n$. Let $W$ be the linear space orthogonal to all $v_i$. Then a set $F$ of the form
\begin{equation}\label{eq:period_strip}
	F = x + [0,1)v_1 + \cdots + [0,1)v_k + W,
\end{equation}
where $x \in \R^n$, is called a \emph{period strip} of $f$.

Consider now the automorphic case. Given a finite index cocompact translation pair $(G, V)$ of a discrete $\Gamma \leqslant E(n)$, the set of translation vectors $\tran_V(G) = \left\{ \tran_V(\gamma) : \gamma \in G \right\}$ is a discrete subgroup of $\R^n$ and spans a linear copy of the affine space $V$. Hence, similarly to above, there exists a free generating set $\{v_1, \ldots, v_k\}$ of $\tran_{V}(G)$. We call a set $F$ a \emph{twisted period strip of $f$ with respect to $(G, V)$} if it is of the form given in \eqref{eq:period_strip} where $x \in \R^n$ and $W$ is the linear space orthogonal to $V$. Moreover, we say that $F$ is a \emph{twisted period strip of $f$} if it is such with respect to some finite index cocompact translation pair of $\Gamma$.

Heuristically, a twisted period strip $F$ of $f$ with respect to $(G, V)$ acts like a regular period strip of $f\vert_V$, but outside $V$ the periods may be twisted around $V$ in ways similar to the group of screw-motions in Section \ref{subsect:translation_subgroup}. We also note that in a coordinate change by a $L \in E(n)$, we may map a twisted period strip $F$ to $F' = LF$, and this is a twisted period strip of $f' = f \circ L^{-1}$ with respect to the finite index cocompact translation pair $(G', V') = (L G L^{-1}, LV)$ of the group $\Gamma' = L \Gamma L^{-1}$.

The twisted period strip essentially acts in our arguments as a fundamental cell replacement which is geometrically simpler and closer to the definition used in the corresponding proofs for periodic functions. The following lemma lets us convert results for twisted periodic strips to corresponding ones for fundamental cells.

\begin{lemma}\label{lem:cover_by_fcells}
	Let $M$ be a closed oriented Riemannian $n$-manifold, and let $f \colon \R^n \to M$ be automorphic under a discrete $\Gamma \subset E(n)$. Let $F$ be a twisted period strip of $f$. Then $F$ can be covered by finitely many fundamental cells of $\Gamma$.
\end{lemma}
\begin{proof}
	Let $(G, V)$ be the corresponding finite index cocompact translation pair of the twisted period strip $F$. Since $G$ is of finite index in $\Gamma$, it suffices to cover $F$ by finitely many fundamental cells of $G$. Moreover, we may assume that $V$ is linear by a coordinate transformation, in which case $\tran_V(\gamma) = \tran(\gamma)$ for all $\gamma \in G$.
	
	Let $G_0 = \{\gamma \in G : T(\gamma) = 0\} \leqslant G$. Then $G_0$ is finite, since otherwise we could use Lemma \ref{lem:basic_prop_of_O(n)} to find an accumulation point of the discrete group $G$. Every element of $G_0$ is identity on $V$, and therefore by Lemma \ref{lem:distance_preserving_fact}, $G_0$ acts on $F$. Moreover, if $\gamma, \gamma' \in G$ with $\tran(\gamma) = \tran(\gamma')$, we have $\gamma' \circ \gamma^{-1} \in G_0$. Therefore, a fundamental cell of the action of $G_0$ on $F$ is also a fundamental cell of $G$. Since $G_0$ is finite, $F$ partitions to finitely many fundamental cells under $G_0$, and the claim follows.
\end{proof}

With the necessary terminology defined, we first recall the original statements of Martio in \cite[Lemma 3.1]{Martio1975paper_kperiod} and by Martio and Srebro in \cite[Theorem 8.3]{MartioSrebro1975paper1}, which we afterwards adapt to our situation. In what follows, the case $0 < k < n-1$ is by Martio and the case $k = n-1$ by Martio--Srebro.

\begin{lemma}[{\cite[Lemma 3.1]{Martio1975paper_kperiod} and \cite[Theorem 8.3]{MartioSrebro1975paper1}}]\label{lem:period_strip_limit_original}
	Let $f \colon \R^n \to \R^n$ be $K$-quasiregular. Suppose that $f$ is $k$-periodic for some $0 < k < n$, and that $f$ has finite multiplicity in a period strip $F$.
	\begin{itemize}
		\item If $k < n-1$, then
		\[
			\lim_{\substack{x \in F\\x \to \infty}} f(x) = \infty.
		\]
		
		\item If $k = n-1$, then there exist $a, a' \in \R^n \cup \{\infty\}$ for which
		\begin{align*}
			\lim_{\substack{x \in F\\x \to +\infty}} f(x) = a,
			&& \lim_{\substack{x \in F\\x \to -\infty}} f(x) = a',
		\end{align*}
		where the limits $+\infty$ and $-\infty$ are defined in terms of an identification $F \cong \R \times D$ where $D$ is bounded. Moreover, either $a = \infty$ or $a' = \infty$.
	\end{itemize}
\end{lemma}

A version of this Lemma for automorphic quasiregular maps into closed manifolds reads as follows.

\begin{lemma}\label{lem:period_strip_limit}
	Let $M$ be a closed oriented Riemannian $n$-manifold, and let $f \colon \R^n \to M$ be $K$-quasiregular. Suppose that $f$ is automorphic under a discrete isometry group $\Gamma \subset E(n)$ of dimension $\dim \Gamma < n$, and $f$ has finite multiplicity in a twisted period strip $F$ with respect to a finite index cocompact translation pair $(G,V)$.
	
	\begin{itemize}
		\item If $\dim \Gamma < n-1$, then there exists $a \in M$ for which
		\[
			\lim_{\substack{x \in F\\x \to \infty}} f(x) = a.
		\]
		\item If $\dim \Gamma = n-1$, then there exist $a, a' \in M$ for which
		\begin{align*}
			\lim_{\substack{x \in F\\x \to +\infty}} f(x) = a,
			&& \lim_{\substack{x \in F\\x \to -\infty}} f(x) = a',
		\end{align*}
		where the limits $+\infty$ and $-\infty$ are defined in terms of an identification $F \cong \R \times D$ where $D$ is bounded.
	\end{itemize}
\end{lemma}

It was already pointed out by Martio in \cite[Section 5.3]{Martio1975paper_kperiod} that the ideas of Lemma \ref{lem:period_strip_limit_original} also work for $f \colon \R^n \to \S^n$, although instead of the limit being $\infty$ it is just some point $a \in \S^n$. In fact, changing the target into a closed oriented Riemannian $n$-manifold causes no significant change in the proofs. Due to Lemma \ref{lem:translation_subgroup_special_case}, case $\dim \Gamma = n-1$ is in fact a direct consequence of the proof of Martio and Srebro. 

The proof of case $0 < \dim \Gamma < n-1$ follows along Martio's proof of \cite[Lemma 3.1]{Martio1975paper_kperiod}, where we only need to make one small adjustment. We however provide the full proof for clarity instead of merely pointing out this difference. We remark that we have included the case $\dim \Gamma = 0$ in the lemma. This is a basic extension result for quasiregular maps of finite multiplicity to an isolated singularity, which should be familiar to experts. Nevertheless, we present its proof here for convenience, as the proof uses some of the same methods as case $0 < \dim \Gamma < n-1$ and is therefore short to present along it.

\begin{proof}[Proof of Lemma \ref{lem:period_strip_limit}]
	We may assume that $V$ is linear by an orientation preserving isometric coordinate change. Let $W = V^\perp$, and let $N = N(f, F)$ be the multiplicity of $f$ in $F$, which was assumed to be finite. We may also assume that the interior of $F$ contains $W$.
	
	We begin by presenting the proof of case $0 < \dim \Gamma < n-1$ in detail. The proof follows that of \cite[Lemma 3.1]{Martio1975paper_kperiod} and is done in the following steps.
	\begin{enumerate}
		\item \label{enum:Martio_step_1} We find an open set $U \subset M$ and a ball $B = B(V, r)$ around $V$ for which $f^{-1} U \subset B$.
		\item \label{enum:Martio_step_2} We use the above $B$ and $U$ to find an annulus $Q$ around $V$ on which the sequence of functions $f_m \colon x \mapsto f(mx)$ has a subsequence converging to some $f_0 \colon Q \to M$.
		\item \label{enum:Martio_step_3} We show that $f_0$ is constant, and let $a$ be its constant value.
		\item \label{enum:Martio_step_4} We prove that the limit of $f$ along $F$ is $a$.
	\end{enumerate}
	The small adjustment compared to the original proof is in step \eqref{enum:Martio_step_3}.
	
	\emph{Step \eqref{enum:Martio_step_1}:} There exists $y \in f(F)$ for which $f^{-1}\{y\} \cap F = \{x_1, \ldots, x_N\}$. We may select bounded normal neighborhoods $U_i'$ of $x_i$, $i \in \{1, \ldots, N\}$, for which $G U_i' \cap U_j' = \emptyset$ when $i \neq j$. Denote $U' = \cup_i U_i'$ and $U = \cap_i fU_i'$.
	
	Now, select $B = B(V, r)$ large enough that $U' \subset B$. We now show that $f^{-1} U \subset B$. Suppose towards contradiction that there exists $x' \notin B$ for which $f(x') \in U$. Then we may first assume $x' \in F$, and then by normality of $U_i'$ we select $x_1', \ldots, x_N' \in GU' \cap F$ for which $f(x_i') = f(x')$. Since $x_i' \in B$ for all $i \in \{1, \ldots, N\}$, we have $x' \neq x_i'$ for all $i \in \{1, \ldots, N\}$. This contradicts the fact that $f$ has a multiplicity of $N$ in $F$.
	
	\emph{Step \eqref{enum:Martio_step_2}:} Let $Q = (3B \setminus 2\overline{B}) \cap \intr F$. Since $\dim V < n-1$, $Q$ is connected. Now, if $f_m \colon Q \to M$ is given by $f_m(x) = f(mx)$, the family $\{f_m\}$ is normal since it consists of $K$-quasiregular maps which omit $U$: this is given e.g.\ in \cite[Proposition 6.1]{OkuyamaPankka2014paper} when the domain is a ball, and we may apply it for $Q$ by covering it with finitely many balls that do not meet $B$. Hence, there exists a subsequence $f_{m_j}$ which converges locally uniformly to a quasiregular map $f_0 \colon Q \to M$. 
	
	\emph{Step \eqref{enum:Martio_step_3}:} We now wish to show that $f_0$ is a constant map. For this, fix a point $x_0 \in W \cap Q$. We consider a compact annulus $A$ of the form $(\overline{A_V} \times W) \cap ((8/3)\overline{B} \setminus (7/3)B)$, where $A_V \subset V$ is an open subset of $V$ for which $0 \notin U_V$. We note that for sufficiently large $m$, the affinely scaled set $mA_V$ must necessarily be larger than $F \cap V$. Hence, by Lemma \ref{lem:distance_preserving_fact}, we find for sufficiently large $j$ points $x_j \in A$ for which $f(m_j x_j) = f(m_j x_0)$. By compactness of $A$, there exists $x_0' \in A$ for which $f_0(x_0') = f_0(x_0)$.
	
	We may now repeat the above construction of $x_0'$ for infinitely many annuli $A$, where the sets $A_V$ are disjoint and accumulate to $0$. By this procedure, we can obtain an accumulation point for $f_0^{-1}\{f_0(x_0)\}$. Hence, $f_0$ cannot be discrete, and therefore it is constant. Let $a$ denote the constant value of $f_0$.
	
	We remark here that the use of an annular region $A$ above is the small adjustment compared to the proof of the periodic case in \cite[Lemma 3.1]{Martio1975paper_kperiod}. The original proof uses a carved-out ball $\overline{B}(x_0, 2r) \setminus B(x_0, r)$ around $x_0$ as the set $A$. However, in our case the points $x_j$ may also be rotated around $V$ instead of merely translated parallel to it, and hence we need $A$ to be an annulus around $V$ in order to replicate the same argument.
	
	\emph{Step \eqref{enum:Martio_step_4}:} It remains to prove that the desired limit holds. Fix a small open ball $B_a \subset M$ around $a$, which is small enough that $U$ is not a subset of $B_a$. Let $E = (5/2)B$ and $F_j = m_{j+1}E \setminus m_j E$. By local uniform convergence of $f_m$, there exists $j_0$ for which $f(\partial (m_{j}E)) \subset B_a$ when $j \geq j_0$. It suffices to show that $f(F_j) \subset B_a$ for all $j \geq j_0$.
	
	Assume to the contrary that $y \in f(F_j) \setminus B_a$. Since $f(F_j)$ omits $U$ which is not contained in $B_a$, there must be infinitely many boundary points of $f(F_j)$ outside $B_a$. Since by Holopainen--Rickman \cite{HolopainenRickman1992paper} the set $M \setminus f(\R^n)$ is finite, we may fix such a boundary point $y'$ for which $y' \in f(\R^n)$.
	
	Now we may select a sequence $(x_i')$ in $F \cap F_j$ so that $f(x_i') \to y'$. Since $F \cap F_j$ is bounded, we find by passing to a subsequence a limit point $x' \in \overline{F_j}$ which $f$ maps to $y'$. Since $y' \in \partial f(F_j)$ and $f$ is open, we obtain that $x' \in \partial F_j = \partial (m_{j+1}E) \cup \partial (m_{j}E)$. Hence, $y' \in f(\partial F_j) \subset B_a$, which is a contradiction, concluding the proof of case $0 < \dim \Gamma < n-1$.
	
	\emph{Other cases:} The proof of case $\dim \Gamma = n-1$ is very similar to the above. Only in this case, $Q$ has two components, and therefore $f_0$ can have two different limit points at opposite directions. Therefore, we leave the details to the interested reader, referring to Martio--Srebro \cite{MartioSrebro1975paper1} for the original proof for periodic functions.
	
	Finally, we quickly prove the case $\dim \Gamma = 0$. We complete step \eqref{enum:Martio_step_1} as in the above proof, where this time $V = \{0\}$ and $W = F = \R^n$. Hence, the resulting $B$ is a ball at origin. We consider $\R^n$ as a subset of $\S^n$ by means of the standard conformal projection, and let $Q = \S^n \setminus \overline{B}$. Then $Q$ is a neighborhood of $\infty$ for which $f(Q \setminus \{\infty\})$ omits the open set $U'$. Hence, $f$ extends quasiregularly to the isolated singularity $\infty \in \overline{\R^n}$ by a standard extension result of quasiregular mappings; see e.g.\ \cite[Theorem 2.6]{OkuyamaPankka2013paper} for a formulation where the target is a manifold.
\end{proof}

We immediately obtain a restriction on the size of $M \setminus f(\R^n)$ as a corollary of Lemma \ref{lem:period_strip_limit}, in the same manner as in Martio--Srebro \cite[Theorem 8.2]{MartioSrebro1975paper1}.

\begin{cor}\label{cor:image_bound}
	Let $M$ be a closed oriented Riemannian $n$-manifold, and let $f \colon \R^n \to M$ be $K$-quasiregular. Suppose that $f$ is automorphic with respect to a discrete isometry group $\Gamma \subset E(n)$ of dimension $\dim \Gamma < n$, and that $f$ has finite multiplicity in a twisted period strip $F$. Then $M \setminus f(\R^n)$ contains at most two points. Moreover, if $\dim \Gamma \neq n-1$, $M \setminus f(\R^n)$ contains at most one point.
\end{cor}

\begin{proof}
	Suppose that $y \in M$ is omitted by $f$. Recall that by Holopainen-Rickman \cite{HolopainenRickman1992paper}, the set $M \setminus f(\R^n)$ contains at most finitely many points. Hence, we may fix a sequence $(f(x_n))$ that converges to $y$. Since $f$ is automorphic with respect to $\Gamma$, we may assume that the sequence $(x_n)$ is in the twisted period strip $F$. 
	
	We first show that the sequence $(x_n)$ converges to infinity. Suppose towards contradiction that it has a bounded subsequence $(x_n')$. Then by moving to a further subsequence we may assume $x_n' \to x' \in \R^n$ as $n \to \infty$. This is a contradction, as $f(x') = y$ but $y$ is an omitted point. Hence, $(x_n)$ converges to infinity. 
	
	Now, by Lemma \ref{lem:period_strip_limit}, the sequence $(f(x_n))$ has a subsequence which converges to either $a$ or $a'$, where $a = a'$ if $\dim \Gamma \neq n-1$. Hence, we have $y \in \{a, a'\}$, which concludes the proof.
\end{proof}

With this, we have now essentially proven all of Theorem \ref{prop:extension_result}. For completeness, we recall the statement and give the final details of the proof.

\begin{customthm}{\ref{prop:extension_result}}
	Let $M$ be a closed connected oriented Riemannian $n$-manifold, let $(\Gamma, h, A)$ be a Latt\`es triple into $M$ with induced Latt\`es map $g \colon h(\R^n) \to h(\R^n)$, and let $k = \dim \Gamma$. Then we have the following three cases: 
	\begin{itemize}
		\item if $k = n$, then $h$ is surjective;
		\item if $k = n-1$, then $h$ omits either 1 or 2 points;
		\item if $0 \leq k \leq n-2$, then $h$ omits 1 point.
	\end{itemize}
	Moreover, $g\colon h(\R^n) \to h(\R^n)$ always extends to a Latt\`es map $g \colon M \to M$ on the entire manifold $M$.
\end{customthm}

\begin{proof}
	If $\dim \Gamma = n$, then $\Gamma$ has a bounded fundamental cell, and it follows that $h(\R^n)$ is a compact open subset of $M$. Hence, $h$ is surjective. If $\dim \Gamma < n$, then $\R^n / \Gamma$ is non-compact, and since $h$ induces a homeomorphism $\R^n / \Gamma \to h(\R^n)$, we have that $h$ omits at least one point. Moreover, the upper bound on the number of points in $M \setminus h(\R^n)$ is given by Corollary \ref{cor:image_bound}. The extension part is due to Lemma \ref{lem:extension_and_non-injectivity}.
\end{proof}


\section{Path lifting construction}\label{sect:path_lifting_short}

The core of Martio's proof of the periodic version in \cite{Martio1975paper_kperiod} is a construction of a path family due to Rickman \cite{Rickman1975paper}. In this section, after introducing the setting and notation, we present the automorphic version of the main lemma \cite[Lemma 4.2]{Martio1975paper_kperiod} of Martio's proof.

\subsection{Notation and preliminary considerations}

We begin by fixing some notation and assumptions for the rest of the paper. Let $M$ be a closed, connected, and oriented Riemannian $n$-manifold, and let $f \colon \R^n \to M$ be $K$-quasiregular. Let $\Gamma \leqslant E(n)$ be a discrete subgroup of Euclidean isometries of dimension $k$, and suppose $f$ is automorphic with respect to $\Gamma$. Let $(G, V)$ be a finite index cocompact translation pair of $\Gamma$. Let $D_G$ be a fundamental cell of $\Gamma$, and let $F$ be a twisted period strip of $f$ with respect to $(G, V)$. Suppose to the contrary of Theorem \ref{thm:generalized_martio_periodicity_result} that $0 < k < n-1$ and $f$ has finite multiplicity in $D_G$, and therefore also finite multiplicity in $F$.

By a change of coordinates, we may assume that $V$ is linear. As before, we denote by $V^\perp$ the linear orthogonal complement of $V$. Let $a \in M$ be the limit obtained in Lemma \ref{lem:period_strip_limit}. By the finite multiplicity of $f$ in $F$, there are only finitely many points in $F$ which $f$ maps to $a$. By another change of coordinates under a translation in $V$, we may assume $f(0) \neq a$. We select a $\delta > 0$ small enough that, if $U$ is the unbounded component of $f^{-1} B(a, \delta)$, then the following hold
\begin{enumerate}
	\item \label{enum:U_def_normal} $U$ is open and connected, $f(U) = B_M(a, \delta) \setminus \{a\}$ and $f (\partial U) = \partial B(a, \delta)$ (ie.\ $U$ is a ``normal domain at infinity around $V$'');
	\item \label{enum:U_def_origin} $f(0) \notin B(a, \delta)$;
	\item \label{enum:U_def_chart} there exists an orientation preserving bilipschitz chart $\varphi_a$ on $B(a, \delta)$ for which $\varphi_a(a) = 0$.
\end{enumerate}
The first condition holds for small enough $\delta$ by a variant of the method in \cite[Lemmas I.4.7--9]{Rickman1993book}. The second condition holds for small enough $\delta$ since we could assume that $f(0) \neq a$. For the final condition, see e.g. \cite[Section 2.3]{Kangaslampi-thesis}.

Now $\varphi_a \circ f$ maps $U$ to a punctured neighborhood $\varphi_a(B_M(a, \delta)) \setminus \{0\}$ of $0$. Let $\iota \colon \R^n \setminus \{0\} \to \R^n$ be the inversion map with respect to the unit ball, and denote $\psi = \iota \circ \varphi_a \circ f$ and $U' = \psi(U)$. Then $\psi \colon U \to U'$ is a quasiregular map which takes the boundary of $U$ to the boundary of $U'$. 

Furthermore, since $f$ is automorphic with respect to $G$, we obtain that $G f^{-1} B(a, \delta) = f^{-1} B(a, \delta)$. Since $U$ is a connected component of $f^{-1} B(a, \delta)$ which is unbounded in distance to $V$, we have for every $g \in G$ that $gU$ is connected and unbounded in distance to $V$. Since $U$ is the only such component of $f^{-1} B(a, \delta)$, we have $gU \subset U$ for every $g \in G$. We conclude that $GU = U$.

\subsection{Automorphic version of the path family construction}

With our current notation and assumptions, the situation now closely resembles the one in Martio \cite{Martio1975paper_kperiod}. The differences are that $k$-periodicity is replaced by automorphicness with respect to $G$, and the domain $U$ is not the entire space $\R^n$.

Before stating the automorphic counterpart of \cite[Lemma 4.2]{Martio1975paper_kperiod}, we recall the concept of modulus of a path family. Suppose that $\Upsilon$ is a family of paths in an open set $U \subset \R^n$. A Borel function $\rho \colon U \to [0, \infty)$ is \emph{admissible} for $\Upsilon$ if
\[
	\int_{\sigma} \rho \geq 1
\]
for every locally rectifiable path $\sigma \in \Upsilon$. The \emph{modulus $M(\Upsilon)$ of the path family $\Upsilon$} is then given by
\[
	M(\Upsilon) = \inf_{\rho} \int_U \rho^n \dd m_n,
\]
where the infimum is over all admissible functions $\rho$ for the family $\Upsilon$. Moreover, suppose the paths of $\Upsilon$ are contained in some sphere $S^{n-1}(r) \subset \R^n$ around the origin. Then we define the \emph{spherical $n$-modulus $M^S_n(\Upsilon)$ of $\Upsilon$} by
\[
	M^S_n(\Upsilon) = \inf_{\rho} \int_{S^{n-1}(r)} \rho^n \dd \cH^{n-1},
\]
where the infimum is again over all admissible functions for $\Upsilon$.

We are now ready to state the following automorphic version of \cite[Lemma 4.2]{Martio1975paper_kperiod}.

\begin{lemma}\label{lem:lifting_result}
	Let $S^{\perp}$ denote the set $\{w \in V^\perp : \abs{w} = 1\}$, let $r_0 > 0$ be such that $\R^n \setminus B^n(r_0) \subset U'$, and let $L_1, L_2 \subset V^\perp$ be half-lines starting from $0$. Then for every $r \geq r_0$, there exists a family of paths $\Upsilon_r$ in $S^{n-1}(r)$ satisfying the following.
	\begin{itemize}
		\item $M^{S}_n(\Upsilon_r) \geq C/(N(f, F)^{n+1}r)$, where $C = C(n)$ is a constant.
		\item Every $\sigma \in \Upsilon_r$ has a $\psi$-lift $\sigma'$ in $U$ which starts from $L_1$ and ends in $GL_2 \cup \{\infty\}$.
	\end{itemize}
\end{lemma}

The proof is long, technical, and nearly unchanged from the periodic case. Hence, we only give its key details, and otherwise refer to Martio \cite{Martio1975paper_kperiod} and Rickman \cite{Rickman1975paper}.

\subsection{Proof of Lemma \ref{lem:lifting_result}}

Let $\pi_G$ be the projection $\R^n \to \R^n / G$, and denote $U_G = \pi_G(U)$. Note that while in the periodic case the space $\R^n / G$ is a manifold, in our more general automorphic case it is an orbifold. However, this difference causes no significant changes in the following proofs compared to the periodic case. 

Since $GU = U$, we have $\pi_G^{-1} U_G = U$, and we may therefore consider $U_G$ as a quotient space $U / G$. Since $f$ is automorphic under $G$, the map $\psi = \iota \circ \phi_a \circ f$ descends to a quotient map $\psi_G \colon U_G \to U'$ in $\pi_G$. The map $\psi_G$ inherits many of the topological properties of $\psi$ such as openness and discreteness. Note that $N(\psi_G, U_G) \leq N(f, D_G) \leq N(f, F) < \infty$.

Suppose $r > 0$ is such that the sphere $S^{n-1}(r) \subset \R^n$ is contained in $U'$. Then, given a path $\alpha \colon [a,b] \to S^{n-1}(r)$, we may locally lift $\alpha$ in the map $\psi_G$. Furthermore, since $GU = U$, we obtain that $\pi_G(\partial U) = \partial U_G$. Hence, by property \eqref{enum:U_def_normal} of the definition of $U$, we see that the maximal lifts of $\alpha$ do not tend to the boundary of $U_G$, and are therefore defined on the entire interval $[a,b]$.

We consider then open spherical caps $C(x, \theta) \subset S^{n-1}(r)$ of angle $\theta$ around $x \in S^{n-1}(r)$, where $C(x, \pi) = S^{n-1}(r)$. We also use the notation $\overline{C}(x, \theta)$ for closed spherical caps, and $\partial C(x, \theta)$ for the boundaries of spherical caps. The first key part of the construction is a boundary path-lifting result for these caps. The following is a version of \cite[Lemma 4.4]{Martio1975paper_kperiod}, and the proof is essentially the same.

\begin{lemma}\label{lem:boundary_lifting}
	Let $x \in S^{n-1}(r)$, $\theta \in (0, \pi)$, and let $\beta \colon [a,b] \to \overline{C}(x, \theta)$ be a path for which $\beta(a) \in \partial C(x, \theta)$ and $\beta\big((a,b]\big) \subset C(x, \theta)$. Suppose that $E$ is a component of $\psi_G^{-1}C(x, \theta)$ and $z \in \psi_G^{-1}\{ \beta(a) \} \cap \overline{E}$. Then there exists a maximal $\psi_G$-lift $\alpha \colon [a,b] \to \overline{E}$ of $\beta$ for which $\alpha(a) = z$ and $\alpha\big((a,b]\big) \subset E$.
\end{lemma}

Next, given $x \in S^{n-1}(r)$, $\theta \in (0, \pi]$ and $z \in \psi_G^{-1}\{x\}$, we denote by $D(z, \theta)$ the $z$-component of $\psi_G^{-1} C(x, \theta)$, and by $D'(z, \theta)$ the $z$-component of $\psi_G^{-1} \overline{C}(x, \theta)$. By the following version of \cite[Lemma 4.7]{Martio1975paper_kperiod}, the set $D'(z, \theta)$ decomposes into closures of sets of the form $D(z', \theta)$; the proof is again essentially the same as in the periodic version.

\begin{lemma}\label{lem:cap_preimage_closure_properties}
	Let $x \in S^{n-1}(r)$, $\theta \in (0, \pi]$ and $z \in \psi_G^{-1}\{x\}$. Then $D'(z, \theta) = \bigcup_{z' \in Z} \overline{D(z', \theta)}$, where $Z = \psi_G^{-1}\{x\} \cap D'(z, \theta)$.
\end{lemma}

Afterwards, it is shown that if $z \in \psi_G^{-1} \{x\}$ and $y \in D(z, \theta)$ with $y \neq z$, then we can shrink the angle $\theta$ until either $y$ is on the boundary of $D(z, \theta)$, or $y$ and $z$ are in different components of $\psi_G^{-1} C(x, \theta)$ but in the same component of $D'(z, \theta)$. This is a part of Martio's proof of \cite[Lemma 4.2]{Martio1975paper_kperiod}, and we again refer there and to Rickman's version \cite[Lemma 3.6]{Rickman1975paper} for the method of proof.

\begin{lemma}\label{lem:angle_selection_lemma}
	Let $x \in S^{n-1}(r)$, $\theta \in (0, \pi]$ and $z \in \psi_G^{-1}\{x\}$. Suppose that $y \in D(z, \theta)$ and $y \neq z$. Let
	\[
		\theta_{z,y} = \sup \{ \tau \in (0, \pi] \colon y \notin C_G(z, \tau) \}.
	\]
	Then the following conditions hold:
	\begin{enumerate}
		\item \label{enum:angle_cond_1} $0 < \theta_{z,y} < \theta$,
		\item \label{enum:angle_cond_2} $y \notin D(z, \theta_{z,y})$, and
		\item \label{enum:angle_cond_3} $y \in D'(z, \theta_{z,y})$.
	\end{enumerate}
\end{lemma}

Next, we consider a specific class of paths on $S^{n-1}(r)$. Let $x, b \in S^{n-1}(r)$ with $x \neq b$, let $S_+$ denote the upper hemisphere of $S^{n-2}$ centered around the basis vector $e_{n-1}$, and let $\nu \colon S^{n-1} \setminus \{e_n\} \to \R^{n-1}$ be the standard conformal projection. We define paths $\beta_{b, x, v} \colon I \to S^{n-1}(r)$ from $b$ to $x$, where $v$ ranges over $S_+$. Suppose first that $x = e_n$ and $\nu(b) \in [0,\infty)e_{n-1}$. In this case, for every $v \in S_+$, we let $\beta_{b, x, v}$ be the path which projects to the half-line $\nu(b) + [0, \infty)v$ in $\nu$. Then, for general $x$ and $b$, we define the paths $\beta_{b, x, v}$ using a rotation to the previous case. We refer to Rickman \cite[pp.\ 801--802]{Rickman1975paper} for a more detailed description of the construction of $\beta_{b, v}$.

The main step of the construction is to find a finite set of anchor points $\{b_1, b_2, \ldots, b_N\} \subset S^{n-1}(r)$ such that, for every $v \in S_+$, we may connect any two points of $\psi_G^{-1}\{x\}$ in the same component of $\psi_G^{-1}S^{n-1}(r)$ by lifts of $\beta_{b_i, x, v}$. We state and discuss this part in more detail, as the following lemma is not directly given in \cite{Martio1975paper_kperiod} or \cite{Rickman1975paper}, but instead described during the process of constructing the family $\Upsilon_r$.

\begin{lemma}\label{lem:formalized_lift_selection}
	Let $x \in S^{n-1}(r)$, let $E$ be a component of $\psi_G^{-1} S^{n-1}(r)$, and let $Z = \psi_G^{-1}\{x\} \cap E$. Let $\cG_Z$ be the complete graph with vertex set $Z$, and denote its set of edges by $\cE_Z$. Then there exist
	\begin{itemize}
		\item a map $P \colon \cE_Z \to S^{n-1}(r)$,
		\item for every $v \in S_+$, a subgraph $\cG_{Z,v}$ of $\cG_Z$,
	\end{itemize}
	which satisfy the following conditions:
	\begin{enumerate}
		\item \label{enum:subtree_1} $\im P$ contains at most $(\# Z) - 1$ points;
		\item \label{enum:subtree_2} every $\cG_{Z,v}$ is a spanning tree of $Z$;
		\item \label{enum:subtree_3} if $\cE_{Z,v}$ is the set of edges of $\cG_{Z,v}$ and $e \in \cE_{Z,v}$ is an edge between $z_1$ and $z_2$, then the path $\beta_{P(e), x, v}$ has lifts $\alpha_1$ and $\alpha_2$ starting from $P(e)$ and terminating at $z_1$ and $z_2$, respectively.
	\end{enumerate}
\end{lemma}

\begin{proof}
	The proof is done by inductively constructing nested subtrees which can be eventually combined to the desired construction. Let $z \in Z$ and $\theta = \pi$. Suppose that $Z \setminus \{z\} \neq \emptyset$. Then for every $w \in Z \setminus \{z\}$, applying Lemma \ref{lem:angle_selection_lemma} yields a $\theta_{z, w} < \theta$. Let $\theta'$ be the maximal such $\theta_{z, w}$.
	
	Now, by the selection of $\theta'$, the set $\psi_G^{-1} C(x, \theta')$ has multiple components which intersect $Z$. Denote these components $E_i$ with $i \in I$, and denote $Z_i = Z \cap E_i$. Moreover, Lemma \ref{lem:cap_preimage_closure_properties} and the selection of $\theta'$ yield that $\bigcup_i \overline{E_i}$ is connected. Let $\cG_1$ be the graph with set of vertices $\{E_i: i \in I\}$ and set of edges $\{(E_i, E_j) : i, j \in I, i \neq j, \overline{E_i} \cap \overline{E_j} \neq \emptyset\}$. It follows that $\cG_1$ is a connected graph with more than one vertex.
	
	Let $\cG_1'$ be a maximal subtree of $\cG_1$. For every edge $(E_i, E_j)$ in $\cG_1'$, there exists a point $b_{i,j} \in \partial E_i \cap \partial E_j$. By Lemma \ref{lem:boundary_lifting}, we may for each $v \in S_+$ select two lifts of $\beta_{b_{i,j}, v}$, with starting point $b_{i,j}$ and endpoints $z_{v} \in Z_i$ and $z_v' \in Z_j$, respectively. We add the edge between $z_v$ and $z_v'$ to $\cG_{X,v}$, and set $P$ to map that edge into $b_{i,j}$.
	
	This construction is then recursively repeated for each $Z_i$ that is not a singleton, starting with some point $z \in Z_i$ and the angle $\theta'$. Since the construction yields more than one $Z_i$, we eventually terminate at a situation where every $Z_i$ is a singleton. It is now easy to see that by the end of the construction, the graph $\cG_{Z,v}$ is a spanning tree of $Z$ for every $v \in S_+$, and in the process we selected at most $(\# Z) - 1$ points $b_{i,j}$.
\end{proof}

We obtain the following corollary.

\begin{cor}\label{cor:constructed_path}
	Let $y$ and $z$ be in the same component $E$ of $\psi_G^{-1} S^{n-1}(r)$, and let $x = \psi_G(z)$. Then there exists a set $X_{y,z} \subset S^{n-1}(r)$ with $\# X_{y,z} \leq N(f, D_G)$, and for every $v \in S_+$, a path in $S^{n-1}(r)$ of the form
	\[
		\beta_v 
		= \beta_{x_{0,v}, x, v}
			* \overleftarrow{\beta_{x_{1,v}, x, v}} 
			* \beta_{x_{1,v}, x, v} 
			* \overleftarrow{\beta_{x_{2,v}, x, v}}
			* \cdots * \beta_{x_{p,v}, x, v},
	\]
	where $p \leq N(f, D_G)$, $x_{i,v} \in X_{y,z}$ for all $i$, and $\beta_v$ has a $\psi_G$-lift from $y$ to $z$.
\end{cor}
\begin{proof}
	Let $x_0 = \psi_G(y)$. We apply Lemma \ref{lem:formalized_lift_selection} to $x$ and $E$ obtaining the map $P \colon \cE_X \to S^{n-1}(r)$, and select $X_{y,z} = \{x_0\} \cup \im P$. Now, suppose $v \in S_+$. We select $x_{0,v} = x_0$, and note that $\beta_{x_{0,v}, v}$ has a $\psi_G$-lift connecting $y$ to some $z' \in \psi_G^{-1}\{x\}$. The statement of Lemma \ref{lem:formalized_lift_selection} now yields the remaining part of $\beta_v$ connecting $z'$ to $z$.
\end{proof}

Before concluding the proof of the construction, it remains to show a property of the components of $\psi_G^{-1} S^{n-1}(r)$.

\begin{lemma}\label{lem:separation_lemma}
	Suppose that $r \geq r_0$. Then there exists a component $D$ of $\psi_G^{-1} S^{n-1}(r)$ such that, for every half-line $L \subset V^\perp$ starting from $0$, we have that $\pi_G(L)$ intersects $D$.
\end{lemma}

For the proof, we recall the \emph{Phragm\'en-Brouwer property} of $\R^n$: if $x, y \in \R^n$, $n \geq 2$, and $A, B \subset \R^n$ are disjoint closed sets for which $A \cup B$ separates $x$ from $y$, then $A$ or $B$ separates $x$ from $y$. This is proven by a simple application of the Mayer-Vietoris sequence of homology groups. Indeed, applying Mayer-Vietoris for $\R^n \setminus A$ and $\R^n \setminus B$, we have the exact sequence
\[\begin{tikzcd}[cramped, sep=small]
	\cdots \ar[r] 
	& H_1(\R^n) \ar[r] 
	& H_0(\R^n \setminus (A \cup B)) \ar[r, "{(i_*,j_*)}"] 
	&[3mm] H_0(\R^n \setminus A) \oplus H_0(\R^n \setminus B) \ar[r] 
	& \ldots 
\end{tikzcd}\]
where $i \colon \R^n \setminus (A \cup B) \hookrightarrow \R^n \setminus A$ and $j \colon \R^n \setminus (A \cup B) \hookrightarrow \R^n \setminus B$ are inclusions. Now, since $H_1(\R^n) = 0$, exactness implies that either $i_*[\{x\}-\{y\}] \neq 0$ or $j_*[\{x\}-\{y\}] \neq 0$, which proves that the property holds.

\begin{proof}[Proof of Lemma \ref{lem:separation_lemma}]
	It is easy to see by a path-lifting argument that every component of $\psi_G^{-1} S^{n-1}(r)$ is mapped surjectively onto $S^{n-1}(r)$. It follows that $\psi_G^{-1} S^{n-1}(r)$ has at most $N(\psi_G, U_G) \leq N(f, D_G)$ many components in $\R^n / G$. Denote these components $D_1, \ldots, D_m$, and let $E_i = \pi_G^{-1} D_i$ for every $i \in \{1, \ldots, m\}$. Note that the sets $E_i$ are closed and disjoint, and that $\bigcup_{i=1}^m E_i = \psi^{-1} S^{n-1}(r)$.
	
	Let then $z$ be a point in the unbounded component of $\psi^{-1} (\R^n \setminus \overline{B^n}(r))$. We then have that $\psi^{-1} S^{n-1}(r)$ separates 0 and $z$. By inductively applying the Phragm\'en-Brouwer property, we find a $E_i$ which separates 0 and $z$. Now, for every half-line $L \subset V^\perp$, the map $\psi$ tends to infinity along $L$, and therefore $L$ meets the unbounded component of $\psi^{-1} (\R^n \setminus \overline{B^n}(r))$. It follows that $L$ meets $E_i$, and therefore a selection of $D = D_i$ completes the proof of the lemma.
\end{proof}

Now, let $r \geq r_0$ and let $L_1$ and $L_2$ be half-lines in $V^\perp$ starting from $0$. It follows by Lemma \ref{lem:separation_lemma} that $\pi_G(L_1)$ and $\pi_G(L_2)$ intersect the same component of $\psi_G^{-1} S^{n-1}(r)$. Let $y$ and $z$ be the respective points of intersection.

We let $\Upsilon_r = \{\beta_v : v \in S_+\}$, where $\beta_v$ are given by Corollary \ref{cor:constructed_path} for $y$ and $z$. Since the paths $\beta_v$ have a lift in $\psi_G$ from $y$ to $z$, it follows that $\beta_v$ also have a lift in $\psi$ which starts from $L_1$ and either ends at $GL_2$ or escapes to infinity. Therefore, the only remaining part of the proof of Lemma \ref{lem:lifting_result} is to prove the modulus estimate. For that, we refer to Rickman's proof in \cite[pp.\ 804--805]{Rickman1975paper}, which yields that $M^{S}_n(\Upsilon_r) \geq C(n) / (\# X_{y,z})$ where $X_{y,z}$ is given by Corollary \ref{cor:constructed_path}.

With that, the proof of Lemma \ref{lem:lifting_result} is complete.


\section{Length of paths}\label{sect:length_estimate}

The core idea behind the proof of the main result is that the lifts $\sigma'$ in Lemma \ref{lem:lifting_result} are long. Martio achieves this in the periodic case by selecting $L_1$ and $L_2$ with opposite directions. Then every translation of $L_2$ by a period moves it further from $L_1$, resulting in a simple length estimate of $\len(\sigma') \geq \abs{\sigma'(0)}$. 

However, since in the automorphic case the maps $\gamma \in G$ also have a rotational component, the automorphic version of this step requires a more careful analysis. Our goal is to prove the following length estimate for the lifts $\sigma'$ in Lemma \ref{lem:lifting_result}.

\begin{lemma}\label{lem:length_estimate}
	Let $\eps > 0$, $k = \dim \Gamma$ and $l = \dim \Gamma_T$. Then there exist half-lines $L_1, L_2 \subset V^\perp$ starting from $0$ such that the following holds: if $\sigma'$ is a path from $L_1$ to $GL_2$ and $s = \inf_t d(V, \sigma'(t))$, then we have 
	\[
		\len(\sigma') \geq C s^{\frac{n-k-1}{n-l-1} - \eps},
	\]
	whenever $s \geq 1$, where $C = C(G, V, \eps) > 0$.
\end{lemma}

The estimate follows from two lemmas. The first one is the asymptotic growth estimate of $\ort(G)$ given in Lemma \ref{lem:orth_part_growth_estimate}. The second one is the following lower bound on the growth of density of a sequence in $S^m$.

\begin{lemma}\label{lem:sphere_filling_estimate}
	Let $(x_j) \in S^m$ be a sequence and let $\eps > 0$. Then there exists a point $y \in S^m$ for which
	\[
		d(y, x_j) \geq \frac{C}{j^{\frac{1}{m}+\eps}},
	\]
	for every $j \in \Z_+$, where $C = C(m, \eps) > 0$.
\end{lemma}

\begin{proof}
	Let $C \in (0, 1)$. We define the set
	\[
		E = \left\{ y \in S^m : 
			d_{S^m}(y, x_j) < Cj^{-1/m-\eps} \text{ for some } j \in \Z_+ \right\},
	\]
	and note that $E = \bigcup_{j \in \Z_+} B_{S^m}(x_j, Cj^{-1/m-\eps})$. Recall that for every $x \in S^m$ and $r \in [0,1]$, we have $\cH^m(B_{S^m}(x, r)) \geq C'r^m$, where $C'$ is dependent only on $m$. Hence, 
	\[
		\cH^m(E) \leq C' C^m \sum_{j=1}^\infty \frac{1}{j^{1+m\eps}} < \infty.
	\]
	Therefore, by selecting $C$ small enough we have that $\cH^m(E) < \cH^m(S^m)$, and any point $y \in S^m \setminus E$ satisfies the required condition.
\end{proof}

Lemma \ref{lem:length_estimate} now follows by assembling these two estimates together.

\begin{proof}[Proof of Lemma \ref{lem:length_estimate}]
	We first let $L_2$ be any half-line in $V^\perp$ starting from $0$. Then, we order the lines $\gamma L_2$ for $\gamma \in G$ with respect to the absolute value of $\tran_V(\gamma)$. The directions of these lines correspond to a sequence $(x_j)$ of points in $S^\perp$, ordered in the same manner with repeat directions skipped. Since $S^\perp$ is an isometric copy of $S^{n-k-1}$, we may apply Lemma \ref{lem:sphere_filling_estimate} to the sequence $(x_j)$, which yields a point $y \in S^\perp$. Now, let $L_1 = [0,\infty)y$.
	
	Suppose that $\sigma'$ is a path from $L_1$ to $\gamma(L_2)$ for some $\gamma \in G$. Let $v = \tran_V(\gamma)$, let $r = \abs{v}$, and let $s = \inf_t d(V, \sigma'(t))$. Furthermore, let $r_{\min} = \min \left\{ \tran_V(\gamma') : \gamma' \in G, \tran_V(\gamma') \neq 0 \right\}$. By discreteness of $G$ we have $r_{\min} > 0$. 
	
	Consider first the case $r > 0$, in which case $r \geq r_{\min}$. We obtain two lower bounds for the length of $\sigma'$. The first one is by considering the distance of $L_1$ and $\gamma L_2$ in the $V$-direction, and is simply
	\begin{equation}\label{eq:length_bound_first}
		\len(\sigma') \geq r.
	\end{equation}
	
	The second estimate is based on the distance of $L_1$ and $\gamma L_2$ in the $V^\perp$-direction. Suppose $\gamma L_2$ is in the direction of $x_j \in S^\perp$ as discussed previously. Then $j \leq \Lambda^V_G(r)$. Now, by Lemma \ref{lem:orth_part_growth_estimate}, there exist $C' > 0$ and $r_0 > 0$ dependent on $V$ and $G$ for which we have $\Lambda^V_G(r) \leq C'r^{k-l}$ if $r \geq r_0$. Since $\Lambda^V_G(r)$ has positive integer values, we may in fact assume that $r_0 = r_{\min}$. Hence, for every $\eps' > 0$, Lemma \ref{lem:sphere_filling_estimate} gives us the estimate
	\begin{equation}\label{eq:length_bound_second}
		\len(\sigma') \geq s d_{S^{n-k-1}}(x_j, y) \geq \frac{Cs}{r^{\frac{k-l}{n-k-1} + \eps'}},
	\end{equation}
	where $C = C(G, V, \eps')$.
	
	Hence, we obtain a lower bound on the length of $\sigma'$ which is the maximum of the bounds \eqref{eq:length_bound_first} and \eqref{eq:length_bound_second}. In terms of $r$, the bound \eqref{eq:length_bound_first} is increasing and the bound \eqref{eq:length_bound_second} is decreasing. Therefore, their maximum assumes its smallest value when the two lower bounds are equal, that is,
	\[
		sC = r^{1+\frac{k-l}{n-k-1} +\eps'}.
	\]
	Solving for $r$ and selecting $\eps'$ suitably gives
	\[
		r = C s^{\left(1 + \frac{k-l}{n-k-1}+\eps'\right)^{-1}} = C s^{\left(1 + \frac{k-l}{n-k-1}\right)^{-1} - \eps} = C s^{\frac{n-k-1}{n-l-1} - \eps}.
	\]
	Since one of our lower bounds was $r$, this completes the proof in the case $r > 0$.
	
	Finally, consider the case $r = 0$. In this case, estimate \eqref{eq:length_bound_second} yields the lower bound
	\[
		\len(\sigma') \geq \frac{Cs}{\left(\Lambda_G^V(r_{\min})\right)^{\frac{1}{n-k-1} + \eps'}} = C''s.
	\]
	If $s \geq 1$, this yields an inequality of the desired form since $s \geq s^{\frac{n-k-1}{n-l-1} - \eps}$.
\end{proof}

With Lemma \ref{lem:length_estimate} proven, we obtain the following simple corollary. The method of proof is due to Martio, see \cite[Section 5]{Martio1975paper_kperiod}.

\begin{cor}\label{cor:length_estimate_sup}
	Let $\eps > 0$, $k = \dim \Gamma$ and $l = \dim \Gamma_T$. Then there exist half-lines $L_1, L_2 \subset V^\perp$ starting from $0$ such that the following holds: if $\sigma'$ is a path from $L_1$ to $GL_2$, $s = \inf_t d(V, \sigma'(t))$ and $r' = \sup_t d(V, \sigma'(t))$, then we have 
	\[
		\len(\sigma') \geq C (r')^{\frac{n-k-1}{n-l-1} - \eps},
	\]
	whenever $s \geq 1$, where $C = C(G, V, \eps) > 0$.
\end{cor}
\begin{proof}
	Let $L_1$ and $L_2$ be given by Lemma \ref{lem:length_estimate}, and suppose that $s \geq 1$. Then if $s \geq r'/2$, we immediately obtain the desired estimate since $\len(\sigma') \geq C s^{\frac{n-k-1}{n-l-1} - \eps}$. However, if instead $s \leq r'/2$, then $\len(\sigma') \geq r'/2$, which trivially implies an estimate of the desired form since $r' \geq s \geq 1$. The claim follows.
\end{proof}


\section{Proof of Theorem \ref{thm:generalized_martio_periodicity_result}}\label{sect:main_theorem}

Let $G \leqslant \Gamma$ be as specified in the beginning of Section \ref{sect:path_lifting_short}. We define the \emph{$G$-modulus} $M_G$ of a path family $\Upsilon \subset \R^n$ by
\[
	M_G(\Upsilon) = \inf_{\rho} \int_{D_G} \rho^n \dd m_n,
\]
where $D_G$ is a fundamental cell of $G$, and $\rho$ varies over all $\Upsilon$-admissible functions which are automorphic under $G$. Heuristically, $M_G$ is the standard conformal modulus on the quotient orbifold $\R^n / G$.

In order to follow Martio's proof, we use the following version of Poletskii's inequality. Since a smooth orientation preserving $\Delta$-bilipschitz map between $n$-dimensional oriented Riemannian manifolds is $\Delta^{2n}$-quasiconformal, the proof is essentially similar to Martio's corresponding proof \cite[Theorem 2.2]{Martio1975paper_kperiod} in the periodic case.

\begin{lemma}
	Suppose $N(f, D_G) < \infty$. Recall that $\psi = \iota \circ \varphi_a \circ f$. Then
	\[
		M(\psi \Upsilon) \leq (K \Delta^{2n})^{n-1} M_G(\Upsilon)
	\]
	for any path family $\Upsilon$ in $\R^n$, where $\Delta$ is the bilipschitz constant of $\varphi_a$.
\end{lemma}

Now we may finally complete the proof of Theorem \ref{thm:generalized_martio_periodicity_result}, following again Martio's corresponding steps in the periodic case but replacing the estimates with ones provided by Corollary \ref{cor:length_estimate_sup}. As previously, we begin by recalling the statement.

\begin{customthm}{\ref{thm:generalized_martio_periodicity_result}}
	Let $M$ be a closed oriented Riemannian $n$-manifold, and let $f \colon \R^n \to M$ be quasiregular. Suppose that $h$ is automorphic with respect to a discrete subgroup $\Gamma$ of the isometry group $E(n)$ of $\R^n$, and let $\Gamma_T$ denote the subgroup of translations of $\Gamma$. If $1 \leq \dim \Gamma \leq n-2$ and
	\[
		\frac{\dim \Gamma_T}{\dim \Gamma} > \frac{1}{n-\dim \Gamma},
	\]
	then $f$ has infinite multiplicity in a fundamental cell of $\Gamma$.
\end{customthm}

\begin{proof}
	Suppose to the contrary that $N(f, D_G) < \infty$, which by Lemma \ref{lem:cover_by_fcells} implies also that $N(f, F) < \infty$. Let $k = \dim \Gamma$ and $l = \dim \Gamma_T$.
	
	Suppose that $\R^n \setminus B^n(r_0) \subset U'$. We may then also assume that $r_0$ is such that, for any $r \geq r_0$, the unbounded component of $\psi^{-1} (\R^n \setminus B^n(r))$ does not meet $B(V, 1)$. Indeed, we may select $R > 1$ for which $\partial B(V, R) \subset U$ and $f^{-1}\{a\} \subset B(V, R)$. Then $\psi(\partial B(V, R))$ is a compact subset of $U'$, and a selection of $r_0$ for which $\psi(\partial B(V, R)) \subset B^n(r_0)$ yields the desired property of $r_0$.
	
	For every $r \geq r_0$, let $L_{1,r}$ and $L_{2,r}$ be the half-lines given by Corollary \ref{cor:length_estimate_sup}. Next, let $\Upsilon_r$ be the family of paths given by Lemma \ref{lem:lifting_result} for $L_{1,r}$ and $L_{2,r}$. Finally, let $\Upsilon = \bigcup_{r \geq r_0} \Upsilon_r$, and let $\Upsilon'$ be the corresponding family of lifts.
	
	We note that $M(\Upsilon) \leq M(\psi\Upsilon')$. By reasoning identical to that of Rickman and Martio, the estimate on $M^S_n(\Upsilon_r)$ given by Lemma \ref{lem:lifting_result} yields that $M(\Upsilon) = \infty$. By the Poletskii's inequality for the $M_G$-modulus, we must also have $M_G(\Upsilon') = \infty$. Our goal is to show that $M_G(\Upsilon') < \infty$, thereby arriving at a contradiction.
	
	Let $\sigma' \in \Upsilon'$ be a lift, let $s = \inf_t d(V, \sigma'(t))$, and let $r' = \sup_t d(V, \sigma'(t))$. Note that we now have $r' \geq s \geq 1$, since our selection of $r_0$ was such that the lifts cannot meet $B(V, 1)$. By Corollary \ref{cor:length_estimate_sup},
	\[
		\len(\sigma') \geq C (r')^{\frac{n-k-1}{n-l-1} - \eps}.
	\]
	Consider the function $\rho \colon \R^n \to \R$ given by
	\[
		\rho(x) = 
		\begin{cases}
			\frac{1}{C d(x, V)^{\frac{n-k-1}{n-l-1} - \eps}},
				& \text{if } d(x, V) \geq 1,\\
			0, & \text{if } d(x, V) < 1.
		\end{cases}
	\]
	Then clearly
	\[
		\int_{\sigma'} \rho \geq 1,
	\]
	and hence $\rho$ is admissible for the family $\Upsilon'$. By Lemma \ref{lem:distance_preserving_fact}, $\rho$ is automorphic under $G$, and hence $\rho$ is admissible for the $M_G$-modulus.
	
	However, now
	\begin{align*}
		M_G(\Upsilon') 
		&\leq \int_{D_G} \rho^n \dd m_n
		\leq \int_{F} \rho^n \dd m_n\\
		&\leq C \mathcal{H}^{k}(\proj_V(F)) \mathcal{H}^{n-k-1}(S^{n-k-1}) 
			\int_1^\infty t^{(n-k-1) -n\left(\frac{n-k-1}{n-l-1}\right) + n\eps} \dd t.
	\end{align*}
	This upper bound is finite if
	\[
		n-k-1 -n\left(\frac{n-k-1}{n-l-1}\right) + n\eps < -1,
	\]
	or alternatively, since we may select $\eps$ as small as desired,
	\[
		n-k-n\left(\frac{n-k-1}{n-l-1}\right) < 0.
	\]
	The claim therefore follows, since this inequality is easily seen to be equivalent with the assumed
	\begin{align*}
		\frac{l}{k} &> \frac{1}{n-k}.
	\end{align*}
\end{proof}

\section{Proof of Theorem \ref{thm:Sylvesters_version}}\label{sect:sylvesters_improvement}

In this section, we present the modification to the proof of Theorem \ref{thm:generalized_martio_periodicity_result} which yields Theorem \ref{thm:Sylvesters_version}. The modification was pointed out to us by Eriksson-Bique. As previously, we recall the statement of the theorem.

\begin{customthm}{\ref{thm:Sylvesters_version}}
	Let $M$ be a closed oriented Riemannian $n$-manifold, and let $f \colon \R^n \to M$ be quasiregular. Suppose that $h$ is automorphic with respect to a discrete subgroup $\Gamma$ of the isometry group $E(n)$ of $\R^n$. If $0 < \dim \Gamma < n-2$, then $f$ has infinite multiplicity in a fundamental cell of $\Gamma$.
\end{customthm}

The proof is by replacing Lemma \ref{lem:length_estimate} with the following lemma.

\begin{lemma}\label{lem:length_estimate_orbit_version}
	Let $\eps > 0$ and $k = \dim \Gamma$. Suppose that $k < n - 2$. Then there exist half-lines $L_1, L_2 \subset V^\perp$ starting from $0$ such that the following holds: if $\sigma'$ is a path from $L_1$ to $GL_2$ and $s = \inf_t d(V, \sigma'(t))$, then we have 
	\[
		\len(\sigma') \geq \sqrt{2} s,
	\]
	whenever $s > 0$.
\end{lemma}

Once Lemma \ref{lem:length_estimate_orbit_version} is proven, a similar version of Corollary \ref{cor:length_estimate_sup} follows with exponent 1, and the rest of the proof proceeds as in Martio's original version, as explained in Section \ref{sect:main_theorem}. It therefore remains to prove Lemma \ref{lem:length_estimate_orbit_version}

\begin{proof}
	We note that the matrix parts of $\gamma \in G$ act on $V^\perp$ orthogonally. That is, by fixing an orthonormal basis in $V^\perp$, we have a homomorphism $\kappa \colon G \to O(n-k)$ which takes $\gamma \in G$ to $\ort(\gamma) \vert_{V^\perp}$ written in the selected basis. Since $G$ is abelian, it follows that $\kappa(G)$ is an abelian subgroup of $O(n-k)$. Moreover, the elements of $O(n-k)$ are complex diagonalizable real matrices. 
	
	Recall that a set of diagonalizable matrices is simultaneously diagonalizable if and only if its elements commute with each other. Hence, it follows that we may simultaneously diagonalize $\kappa(G)$. Consequently, $V^\perp$ splits into a direct sum $V^\perp = Q_1 \oplus \dots \oplus Q_l$, where $Q_i$ are $\ort(G)$-invariant subspaces of dimension 1 or 2. Moreover, since orthogonal matrices have pairwise orthogonal eigenspaces, we may assume that $Q_i$ are pairwise orthogonal.
	
	By our assumption $k < n-2$, we have $\dim V^\perp > 2$, so there exists more than one space $Q_i$. Therefore, we may select $L_1$ in a direction contained in $Q_1$, and $L_2$ in a direction contained in $Q_2$. It follows that for all $\gamma \in G$, the half-lines $L_1$ and $\gamma(L_2)$ are orthogonal. Hence, if $\sigma'$ is a path from $L_1$ to $GL_2$, then $\len(\sigma') \geq \sqrt{2} s$.
\end{proof}


\end{document}